\documentclass[twocolumn]{IEEEtran}

\ifCLASSINFOpdf
\usepackage[cmex10]{amsmath}
\usepackage{subfigure}
\usepackage{graphics}
\usepackage{epsfig}
\usepackage{subfig}
\usepackage{multirow}
\usepackage{color}
\usepackage{amsfonts}
\usepackage{epstopdf}
\usepackage{mathptmx}
\usepackage{amsmath}
\allowdisplaybreaks[4]
\usepackage{amssymb}

\newtheorem{assm}{Assumption}

\newtheorem{lema}{Lemma}
\newtheorem{remark}{Remark}
\newtheorem{thme}{Theorem}
\newenvironment*{proof}{{\it Proof.}}{\hfill $\square$\par}
\newtheorem{prope}{Proposition}

\newtheorem{coro}{Corollary}
\begin{document}

\title{Delay-Adaptive Boundary Control
 of Coupled Hyperbolic PDE-ODE Cascade Systems }

\author{Ji~Wang
       and Mamadou Diagne
\thanks{
J. Wang is with the Department of Automation, Xiamen University, Xiamen,
Fujian 361005, China (e-mail:jiwang9024@gmail.com).
}
\thanks{
M. Diagne is with the Department of Mechanical and Aerospace Engineering, University of California San Diego,  9500 Gilman Dr, La Jolla, CA 92093 (mdiagne@ucsd.edu).
}
}

\maketitle

\begin{abstract}
This paper presents a delay-adaptive boundary control scheme for a $2\times 2$ coupled linear hyperbolic PDE-ODE cascade system with an unknown and arbitrarily long input delay. To construct a nominal delay-compensated control law, assuming a known input delay, a three-step backstepping design is used. Based on the certainty equivalence principle, the nominal control action is fed with the estimate of the unknown delay, which is generated from a batch least-squares identifier that is updated by an event-triggering mechanism that evaluates the growth of the norm of the system states. As a result of the closed-loop system, the actuator and plant states can be regulated exponentially while avoiding Zeno occurrences. A finite-time exact identification of the unknown delay is also achieved except for the case that all initial states of the plant are zero. As far as we know, this is the first delay-adaptive control result for systems governed by heterodirectional hyperbolic PDEs. The effectiveness of the proposed design is demonstrated in the control application of a deep-sea construction vessel with cable-payload oscillations and subject to input delay.\end{abstract}


%===============================================================================

\section{Introduction}
\subsection{Boundary control of coupled hyperbolic PDEs}\label{sec:mo}
Systems of transport partial differential equations (PDEs) appear in many physical models, including road traffic \cite{Goatin2006,Yu2019Traffic,Zhang2017Necessary}, water management systems \cite{Diagne2017,Diagne2017b,Prieur2014Boundary, Prieur2008Robust}, {flow of fluids in oil drilling systems \cite{Hasan2016Boundary,Hasan2014Moving}}, and cable vibration dynamics \cite{J2020two,J2018Balancing}. As a result of the backstepping design \cite{Coron2013Local,Vazquez2011Backstepping}, the sliding mode control approach \cite{Lamare2015Control} and the proportional-integral (PI) controller design \cite{Tang2014}, the theoretical results on boundary control of coupled first-order linear hyperbolic PDEs have emerged in the last decade. The backstepping design
was further extended to that of a $n+1$ system in \cite{Meglio2013Stabilization}, and then to  a more general coupled
transport PDE system where the
number of PDEs in either direction is arbitrary  \cite{Hu2016Control}. Along the same lines,  studies on the design of an adaptive estimation framework have been proposed in \cite{Anfinsen2019Adaptivea, Anfinsen2019Adaptiveb} and extended to adaptive control in \cite{Anfinsen2019Adaptive}. However,  the problem of delay-adaptive control for hyperbolic PDEs has gone unanswered in all these last developments as traditional designs based on swapping identifiers, passive identifiers, and Lyapunov functions remain difficult to exploit for such systems.

\subsection{Delay-compensated control of finite- and  infinite-dimensional plants}
Time delays, which are well known to be detrimental to stability \cite{Guo2008},   often exist in practical control systems.
In order to compensate for arbitrarily long delays,  ``avant-garde" backstepping-based delay compensation techniques were first developed in  \cite{krstic2009delay,krstic2008Backstepping}. Bottom-line, the input delay is converted into a transport PDE as an infinite-dimensional representation of the actuator state.   For ODE plants, a PDE/ODE cascade system ensues from this substitute representation of the actuator state.  The method has also been used to compensate for the effect of sensor delays that oftentimes occur in  ODE plants.   In comparison to many results  \cite{Ahmed-Ali2013Global,Cacace2010,Germani2002}, which only estimate
plant states, the approach proposed in \cite{krstic2009delay,krstic2008Backstepping} enables estimation of
both the plant and the sensor states when designing a feedback loop. A number of results considering delays that are described by complex transport actuation paths for nonlinear ODE plants were developed in \cite{Diagnedelaya,Diagnedelayb} and the references therein.

While compensation for arbitrarily long delays is commonly
available for finite-dimensional systems, only very few examples for infinite-dimensional systems were presented, where one pioneering result is \cite{Krstic2009control} that is conceived using backstepping. In recent years, researchers from the PDE control community have shifted their attention to this topic, leading to many interesting developments that can be found in  \cite{Koga2020Delay,Hugoa,Hugob,Jie2020,Jie2019}. By treating the delay as a transport PDE, \cite{Krstic2011} presented the design of a boundary controller for a pure wave PDE with compensation of an arbitrarily long input delay while ensuring exponential stability for the closed-loop system. For coupled heterodirectional hyperbolic PDEs, in \cite{J2020Delay}, a delay-compensated control scheme was designed for a sandwich hyperbolic PDE in the presence of a sensor delay of arbitrary length. {In the same spirit}, \cite{Jie2022} proposed a distributed input delay compensation for traffic systems governed by coupled hyperbolic PDEs (see  \cite{Jie2021} as well). In addition to the continuous-in-time control law, on the basis of the event-triggered boundary control design of PDEs \cite{Espitia2018,Espitia2020,JiEvent2021}, an event-triggered delay-compensated boundary control law for coupled hyperbolic PDEs was presented in \cite{J2022Delay-compensated}. Although the above substantial results emblematize a  major step in the field,  the prior knowledge of the delay length is a mandatory weakening factor that mitigates their viability for many practical applications.

Lyapunov designs have been employed to develop delay-adaptive controllers for linear and nonlinear ODE plants \cite{bekiaris2012Adaptive,Bresch2010,Bresch2014,Zhu2015,Zhu2017,Krstic2009} via backstepping-based certainty-equivalence compensators. The primary idea behind these contributions is to estimate the unknown delay using input-output signals, and then adjust a pre-designed nominal controller based on estimated parameters in order to achieve convergence. In general, compared to other traditional parameter identifier methods like swapping or passive identifiers, the Lyapunov technique provides better transient performance properties.
Recently, the approach has been extended to linear reaction-diffusion PDEs with a boundary or a distributed delayed input \cite{Wang2021b,Wang2021}, where asymptotic convergence results are achieved. {In the realm  of advancing the design approach of \cite{Wang2021b,Wang2021}, a recent result has  achieved the design of Lyapunov-based delay-adaptive boundary control for a scalar Integral PDE \cite{wang2023delay}. As far as we are aware, the three preceding contributions are the sole results on delay-adaptive control for PDE plants.} The method in the present contribution is different with both the above delay-adaptive control results and traditional adaptive control designs for hyperbolic PDEs  \cite{Anfinsen2019Adaptive}. More precisely, our design relies on a triggered batch least-square identifier (BaLSI), a novel approach that was initially introduced in \cite{Karafyllis2019Adaptive,Karafyllis2018Adaptive}, which has at least two significant advantages over the traditional adaptive control approaches: guaranteeing {exponential regulation} of the states to zero, as well as { finite-time convergence} of the estimates to the true values. This method has been applied in adaptive control of a parabolic PDE \cite{Karafyllis2019Adaptive1}, and first-order hyperbolic PDEs in \cite{JiAdaptive2021,J2022Event,JiACC2021} with unknown plant parameters.
\subsection{Contributions}
Comparison with some results on boundary control of PDEs with arbitrarily long input delays is shown in Table \ref{tab:com1}.
\begin{itemize}
\item To the best of our knowledge, our result is the first delay-adaptive
compensator for coupled hyperbolic PDEs involving an unknown and arbitrarily large input delay.

\item Different with the delay-robust stabilizing feedback control design for coupled first-order hyperbolic PDEs that achieve robustness to small delays in actuation \cite{Auriol2018Delay}, the present contribution ensures {exact compensation} of the arbitrarily large unknown input delay.

\item The finite-time exact identification of the unknown delay is achieved. As a result, the exponential regulation, instead of the asymptotic one in \cite{Wang2021b,Wang2021}, is guaranteed in the closed-loop system. {Basically, after the finite-time identification of the unknown parameter, the delay-adaptive control signal is identical to the nominal control action (with known input delay), which ultimately improves substantially the resulting transient performance of the whole closed-loop system's dynamics. }
\end{itemize}
\begin{table}
\centering
\caption{Delay-compensated control of PDEs with arbitrarily long input delays}\label{tab:com1}
\begin{tabular}{|c|c|c|}

  \hline
   & Known delay & Unknown delay \\
    \hline
Parabolic PDEs & \cite{Koga2020Delay}\cite{Krstic2009control}\cite{Hugoa}\cite{Jie2020} &  \cite{Wang2021b}\cite{Wang2021}  \\
 \hline
Hyperbolic PDEs & \cite{Krstic2011}\cite{Jie2021}\cite{Jie2022}\cite{Jie2019} & This paper  \\
  \hline
\end{tabular}
\end{table}
 \subsection{Organization}
The problem formulation is shown in Section \ref{sec:problem}. The nominal control design is presented in Section \ref{sec:nominalcontrol}. The design of delay-adaptive control with piecewise-constant parameter identification is proposed in Section \ref{sec:adaptive}. The main result including the absence of a Zeno phenomenon, parameter convergence,  and exponential regulation of the states is proved in
 Section \ref{sec:sta}.  The effectiveness of the proposed design
is illustrated with a numerical simulation of a deep-sea construction vessel (DCV)  in Section \ref{sec:sim}. The conclusion and
future work are presented in Section \ref{sec:conclusion}.
\subsection{Notation} We adopt the following notation.
\begin{itemize}
\item The symbol $\mathbb Z^+$ denotes the set of natural numbers including zero, and the notation $\mathbb N$ for the set $\{1,2,\cdots\}$, i.e., the natural numbers without $0$. We also use $\mathbb R_+:=[0,+\infty)$.

\item Let $U\subseteq \mathbb R^m$ be a set with non-empty interior and let $\Omega\subseteq\mathbb R$
be a set. By $C^0 (U;\Omega)$, we denote the class of continuous
mappings on $U$, which takes values in $\Omega$. By $C^k (U;\Omega)$, where
$k \ge 1$, we denote the class of continuous functions on $U$,
which have continuous derivatives of order $k$ on $U$ and take
values in $\Omega$.

\item We use the notation $L^2(0, 1)$ for the standard space of the equivalence
class of square-integrable, measurable functions defined
on $(0, 1)$ and $\|f\|=\left(\int_0^1 f(x)^2 dx\right)^{\frac{1}{2}}<+\infty$ for $f \in L^2(0, 1)$.

\item For an $I\subseteq \mathbb R_+$, the space $C^0(I;L^2(0,1))$ is the space of continuous mappings $I\ni t\mapsto u[t]\in L^2(0,1)$.

\item Let $u: \mathbb R_+\times [0,1]\rightarrow \mathbb R$  be given. We use the notation $u[t]$ to denote the profile of $u$ at certain $t\ge 0$, i.e., $(u[t])(x)=u(x,t)$, for all $x\in[0,1]$.
\end{itemize}
\section{Problem Formulation}\label{sec:problem}
Consider  the potentially open-loop unstable plant governed by the following $2\times 2$ linear hyperbolic PDE coupled with  a linear ODE,
\begin{align}
\dot X (t)& = AX (t) + Bw(0,t) ,\label{eq:ode1}\\
{z_t}(x,t) &=  - q_1{z_x}(x,t)+{d_1}z(x,t)+{d_2}w(x,t),\label{eq:ode2}\\
{w_t}(x,t) &= q_2{w_x}(x,t)+{d_3}z(x,t)+{d_4}w(x,t),\label{eq:ode3}
\end{align}
with the boundary conditions:
\begin{align}
z(0,t)  &= CX (t)- pw(0,t),\label{eq:ode4}\\
w(1,t) &=  c_0U(t-D)+qz(1,t),\label{eq:ode5}
\end{align}
where, $q, q_1, q_2, d_1, d_2, d_3, d_4, c_0 $ and $p$ are arbitrary  parameters with $q_1,q_2> 0$ being transport speeds, and $p\neq 0$, $c_0\neq 0$. Here, the matrix  $A, B, C$ are known,  $z(x,t)$ and  $w(x,t)$  are the PDE state variables, $X(t)\in \mathbb R^m$ is the linear ODE state, $U$ is the control variable and $D>0$ is the indiscriminately large and unknown input delay. We assume that the initial conditions satisfy
\begin{align}
z^0(x), w^0(x)  \in L^2(0, 1),\ X^0\in \mathbb{R}^m
\end{align}
and consider the following assumptions.
\begin{assm}\label{as:AB}
The pair $A,B$ is controllable.
\end{assm}
\begin{assm}\label{as:pq}
Parameters $p,q$ satisfy
\begin{align}
{|pq|}{e^{  {\max\left\{\frac{2d_4}{q_2},\frac{2d_1}{q_1}\right\}}}}\le\frac{1}{\sqrt{2}}.
\end{align}
\end{assm}
\begin{assm}\label{as:coe1}
The bounds of the unknown input delay $D$ are  known
and  arbitrary, i.e.,
\begin{align}
0<\underline D\le D\le \overline D
\end{align}
where positive constants $\underline D$, $\overline D$ are arbitrary.
\end{assm}

Our goal is to design a delay-adaptive boundary control action, $U(t)$,  that exponentially regulates the system \eqref{eq:ode1}--\eqref{eq:ode5}  despite the presence of an unknown delay $D$ whose length is arbitrary. {The plant \eqref{eq:ode1}--\eqref{eq:ode5} {can be used to model}  cable-payload oscillations in DCV, which are to be suppressed for the purpose of accurate placement of the equipment to be installed on the sea floor. From this application perspective,  large-distance signal transmission  in the water through a set of acoustics
devices and  the actuation of the  hydraulic actuator for the ship-mounted crane are subject to delays, which are considered as an unknown delay in the control input, the cable vibration dynamics are governed by  the $2\times 2$ hyperbolic PDE, and  the vibration dynamics of the cage are captured by the ODE system.}
\section{Nominal delay-compensated control design}\label{sec:nominalcontrol}
In order to design the nominal control law, we first construct an infinite-dimensional representation of the actuator state by converting the delayed input into transport PDE actuation dynamics. Define a new variable $v(x,t)$ as
 \begin{equation*} v(x,t)= \begin{cases} U\left(t-{D}x\right)&\mbox{if $t-{D}x\ge0$}\\ 0&\mbox{if $t-{D}x<0$} \end{cases} \end{equation*}
then \eqref{eq:ode5} is rewritten as
\begin{align}
w(1,t) &=  c_0v(1,t)+qz(1,t),\label{eq:ode6}\\
v_t(x,t)&=-\frac{1}{D}v_x(x,t),\label{eq:ode7}\\
v(0,t)&=U(t), \label{eq:ode8}\\
v(x,0)&=0\label{eq:ode9}
\end{align}
for $x\in[0,1],t\in[0,\infty)$.
Now, resulting from the new representation of the actuator state, the function $U(t)$,  which is defined as the boundary condition \eqref{eq:ode8} of the transport equation \eqref{eq:ode7},  is the delay-free control input to be designed for the hyperbolic PDE-PDE-ODE cascade system consisting of \eqref{eq:ode1}--\eqref{eq:ode4} combined with \eqref{eq:ode6}--\eqref{eq:ode9}.

\subsection{First Step: Backstepping  Transformation for the $2\times2$ Coupled Hyperbolic PDE-ODE}
We introduce the following backstepping transformation \cite{Meglio2017Stabilization} in order to remove the in-domain coupling  destabilizing terms from  the $2\times 2$ hyperbolic PDE system consisting of \eqref{eq:ode2}, \eqref{eq:ode3} and make the ODE system matrix Hurwitz:
\begin{align}
\alpha (x,t) =&z(x,t) - \int_0^x {\phi}(x,y)z(y,t)dy\notag\\
& -  \int_0^x {\varphi}(x,y)w(y,t)dy -  \gamma (x){X}(t),\label{eq:contran1a}\\
\beta (x,t) =& w(x,t) - \int_0^x {\Psi}(x,y)z(y,t)dy \notag\\
&-  \int_0^x {\Phi}(x,y)w(y,t)dy -  \lambda (x){X}(t)\label{eq:contran1b}
\end{align}
whose inverse is
\begin{align}
z (x,t) =&\alpha(x,t) - \int_0^x {\bar\phi}(x,y)\alpha(y,t)dy\notag\\
& -  \int_0^x {\bar\varphi}(x,y)\beta(y,t)dy -  \bar\gamma (x){X}(t),\label{eq:Icontran1a}\\
w (x,t) =& \beta(x,t) - \int_0^x {\bar\Psi}(x,y)\alpha(y,t)dy \notag\\
&-  \int_0^x {\bar\Phi}(x,y)\beta(y,t)dy -  \bar\lambda (x){X}(t)\label{eq:Icontran1b}
\end{align}
to convert  \eqref{eq:ode1}--\eqref{eq:ode4}, \eqref{eq:ode6} into
\begin{align}
\dot X (t) =& A_{\rm m}X (t) + {B}\beta (0,t), \label{eq:targ5}\\
  \alpha (0,t) =& -p\beta  (0,t),\label{eq:targ2}\\
{\alpha _t}(x,t) =&  - {q_1}{\alpha _x}(x,t) +{d_1}\alpha (x,t),\label{eq:targ1}\\
{\beta _t}(x,t) =& {q_2}{\beta _x}(x,t) + {d_4}\beta (x,t),\label{eq:targ4}\\
\beta(1,t) =&  c_0v(1,t)+q\alpha(1,t)+  \left(\bar\lambda (1) -  q\bar\gamma (1)\right){X}(t)\notag\\
&+ \int_0^1 \left({\bar\Psi}(1,y)-q{\bar\phi}(1,y)\right)\alpha(y,t)dy\notag\\&+ \int_0^1 \left({\bar\Phi}(1,y)-q{\bar\varphi}(1,y)\right)\beta(y,t)dy. \label{eq:targ8}
\end{align}
The gain vector $K$ is selected so that
\begin{align}
A_{\rm m}=A+BK^T \label{eq:Am}
\end{align}
is Hurwitz.

The conditions on the kernels ${\phi}(x,y)$, ${\varphi}(x,y)$, $\gamma (x)$, ${\Psi}(x,y)$, ${\Phi}(x,y)$, $\lambda (x)$ and ${\bar\phi}(x,y)$, ${\bar\varphi}(x,y)$, $\bar\gamma (x)$, ${\bar\Psi}(x,y)$, ${\bar\Phi}(x,y)$, $\bar\lambda (x)$ in the backstepping transformations \eqref{eq:contran1a}--\eqref{eq:Icontran1b}, which are obtained by matching the original system \eqref{eq:ode1}--\eqref{eq:ode5} and the intermediate system \eqref{eq:targ5}--\eqref{eq:targ8},  are shown in the part 1 of Appendix A, and the well-posedness of the kernel conditions has been proved in Theorem 4.1 of \cite{Meglio2017Stabilization}.
\subsection{Second Step: {Transformation of the Actuator States}}
With the purpose of removing the  integral terms and ODE state $X(t)$ from the PDE boundary condition \eqref{eq:targ8}, we define the following change of coordinate
\begin{align}
u(x,t)=&v(x,t)+\int_0^1 K_1(x,y)\alpha(y,t)dy+\int_0^1 K_2(x,y)\beta(y,t)dy\notag\\&+\eta(x){X}(t)\label{eq:secondtr}
\end{align}
which enables one to  map  the actuator dynamics given by \eqref{eq:ode7},  {\eqref{eq:ode8}, and \eqref{eq:targ8} into the following equations}
\begin{align}
\beta(1,t) =&  c_0u(1,t)+q\alpha(1,t), \label{eq:tar9}\\
{u_t}(x,t) =&  - d{u_x}(x,t)+q_2K_2(x,1)c_0u(1,t), \label{eq:tar10}\\
u(0,t)=&U(t)+\int_0^1 K_1(0,y)\alpha(y,t)dy\notag\\&+\int_0^1 K_2(0,y)\beta(y,t)dy+\eta(0){X}(t),\label{eq:tar11}
\end{align}
where $$d=\frac{1}{D}.$$ The detailed computation and conditions of the kernels $K_1(x,y),K_2(x,y),\eta(x)$  are given in the part 2 of Appendix A.
\subsection{Third Step: Backstepping Transformation for the {Resulting $u$-PDE}  }
To remove the boundary nonlocal term $q_2K_2(x,1)c_0u(1,t)$ in the transport PDE  \eqref{eq:tar10}, we apply the following mapping
\begin{align}
u(x,t)=\hat u(x,t)+\int_x^1 R(x,y) \hat u(y,t)dy\label{eq:third}
\end{align}
which  converts \eqref{eq:tar9}--\eqref{eq:tar11} into
\begin{align}
\beta(1,t) =&  c_0\hat u(1,t)+q\alpha(1,t), \label{eq:tar12}\\
{\hat u_t}(x,t) =&  - {d}{\hat u_x}(x,t),\label{eq:tar13}\\
\hat u(0,t)=&0,\label{eq:tar14}
\end{align}
with the nominal control input  defined as
\begin{align}
U(t)=&-\int_0^1 K_1(0,y)\alpha(y,t)dy-\int_0^1 K_2(0,y)\beta(y,t)dy\notag\\&+\int_0^1 R(0,y) \hat u(y,t)dy-\eta(0){X}(t).\label{eq:targetinput}
\end{align}
The conditions of the kernel $R(x,y)$ are shown in  the part 3 of Appendix A. The inverse transformation of \eqref{eq:third} can be found as
\begin{align}
\hat u(x,t)= u(x,t)+\int_x^1 P(x,y) u(y,t)dy,\label{eq:thirdinv}
\end{align}
where the conditions of $P(x,y)$ are given in the part 3 of Appendix A as well.
\begin{figure*}
\centering
\includegraphics[width=18cm]{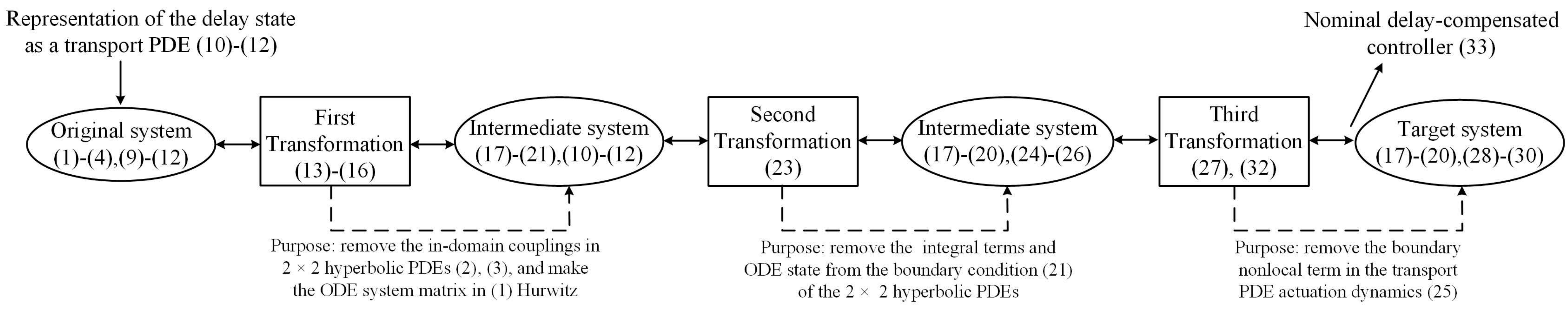}
\caption{The flow diagram of the nominal delay-compensated control design.}
\label{fig:controldesign}
\end{figure*}
\subsection{Stability result of nominal delay-compensated control}
The flow diagram of the nominal delay-compensated control is shown in Figure \ref{fig:controldesign}. In a nutshell,   the prior transformations convert the original system that consists of \eqref{eq:ode1}--\eqref{eq:ode4}, \eqref{eq:ode6}--\eqref{eq:ode9} into the target system that consists of \eqref{eq:targ5}--\eqref{eq:targ4}, \eqref{eq:tar12}--\eqref{eq:tar14}. {The nominal control input \eqref{eq:targetinput} }is rewritten with respect to  the original state variables as follows
\begin{align}
U(t)=&-\int_0^1 M_1(y;D)z(y,t)dy-\int_0^1 M_2(y;D)w(y,t)dy\notag\\&+\int_0^1 M_3(y;D)  v(y,t)dy+M_4(D) {X}(t),\label{eq:U}
\end{align}
where the controller gains $M_i$, $\{1,...,4\}$ are given in Appendix B.  Writing $D$ after $``;"$ in $M_i(y;D)$ emphasizes the fact that these functions are parameterized by the delay $D$.

The stability result of the nominal delay-compensated control is stated as follows.
\begin{thme}\label{theorem:1}
For the known delay $D$, with arbitrary initial data $(z[0],w[0])^T\in L^2(0,1)$, $X(0)\in \mathbb R^m$,  considering the closed-loop system consisting of the plant \eqref{eq:ode1}--\eqref{eq:ode4}, \eqref{eq:ode6}--\eqref{eq:ode9} and the nominal controller \eqref{eq:U}, the exponential stability of the closed-loop system is obtained in the sense that there exist positive constants $\Upsilon,\lambda_1$ such that
\begin{align}
\Omega(t)\le \Upsilon\Omega(0)e^{-\lambda_1 t},~~t\ge 0,\label{eq:thm1}
\end{align}
where $\Omega(t)$ is defined as
\begin{align}
\Omega(t)=\|z[t]\|^2+\|w[t]\|^2+\|v[t]\|^2+ |{X}(t)|^2. \label{eq:omeganorm}
\end{align}
\end{thme}
\begin{proof}
Define Lyapunov function $V(t)$ as
\begin{align}
V(t) =&  \frac{r_d}{2}{X}^T(t)P_1X(t) + \frac{r_a}{2}{}\int_0^1 {{e^{{\delta }x}}} \beta {(x,t)^2}dx \notag\\
&+ \frac{1}{2}{}\int_0^1 {{e^{ - {\delta}x}}} \alpha {(x,t)^2}dx+ \frac{r_c}{2}{}\int_0^1 {{e^{ - {}x}}} \hat u {(x,t)^2}dx,\label{eq:V}
\end{align}
where a positive definite matrix $P_1 = {P_1}^T$ is the solution to the
Lyapunov equation $A_{\rm m}^TP_1+P_1A_{\rm m}=-Q_1$ for some $Q_1={Q_1}^T>0$, and where $\delta,r_a,r_c,r_d$ satisfy
 \begin{align}
&{e^{ - 2{\max\left\{\frac{2d_4}{q_2},\frac{2d_1}{q_1}\right\}}}}>{e^{ - 2{\delta}}}>2{p^2q^2},\label{eq:delta}\\
&\frac{q_1}{2q_2q^2}{e^{ - 2{\delta}}}\ge {r_a}>\frac{p^2q_1}{q_2},\label{eq:ra}\\
&{{r_c}}\ge 2\bar Dq_2 {r_a}{e^{{\delta}}}c_0^2,\label{eq:rc}\\
&0<r_d\le\frac{{\lambda_{\rm min}(Q_1)}}{2|P_1B|^2}\left(q_2 {r_a}-{p^2q_1}\right).\label{eq:rd}
\end{align}
Please note that ${e^{ - 2{\max\left\{\frac{2d_4}{q_2},\frac{2d_1}{q_1}\right\}}}} >2{p^2q^2}$ holds in \eqref{eq:delta} under Assumption \ref{as:pq}, and $\frac{q_1}{2q_2q^2}{e^{ - 2{\delta}}}>\frac{p^2q_1}{q_2}$ holds in \eqref{eq:ra} due to the right inequality in \eqref{eq:delta}, which means the existence of  $\delta,r_a$ satisfying \eqref{eq:delta}, \eqref{eq:ra}. It is then straightforward to obtain $r_c,r_d$ by \eqref{eq:rc}, \eqref{eq:rd}, where the positiveness of the right-hand side of \eqref{eq:rd} is ensured by the right inequality in \eqref{eq:ra}. Therefore, there exists a solution $\delta,r_a,r_c,r_d$ satisfying \eqref{eq:delta}--\eqref{eq:rd}.

Following the calculation in {Appendix C} and using the norm estimate \eqref{equi},  we have
\begin{align}
\xi_1\xi_3\Omega(t)\le V(t)\le \xi_2\xi_4\Omega(t)\label{eq:barOmega}
\end{align}
where $\xi_1,\xi_2$ are given in \eqref{eq:xi1}, \eqref{eq:xi2}, and
\begin{align}
\xi_3&=\frac{1}{2}\min\left\{{r_d}\lambda_{\rm min}(P_1),{r_a}, {{e^{ - {\delta}}}},{r_c} {{e^{ - {1}}}}\right\},\\
\xi_4&=\frac{1}{2}\max\left\{{r_d}\lambda_{\rm max}(P_1),{r_a}{{e^{{\delta }}}},1,{r_c}\right\},
\end{align}
where ${\lambda_{\rm min}(P_1)}$ {is the smallest eigenvalue of $P_1$}.

Taking the derivative of \eqref{eq:V} along \eqref{eq:targ5}--\eqref{eq:targ4}, \eqref{eq:tar12}--\eqref{eq:tar14},
and applying Young's inequality and the Cauchy-Schwarz inequality,  the following estimate holds for all $t\ge 0$:
\begin{align}
\dot V(t)  \le &-\frac{r_d}{4}{\lambda_{\rm min}(Q_1)}|{X}(t)|^2\notag\\&- \bigg(\frac{{1}}{2}q_1{e^{ - {\delta}}}-q_2 {r_a}{e^{{\delta}}}q^2\bigg)\alpha (1,t)^2 \notag\\
&- \left(\frac{1}{2}q_2 {r_a}-\frac{r_d|P_1B|^2}{{\lambda_{\rm min}(Q_1)}}-\frac{p^2q_1}{2} \right)\beta {(0,t)^2}\notag\\&- r_a\left(\frac{1}{2}{\delta}q_2 -d_4\right)\int_0^1 {{e^{{\delta}x}}} \beta {(x,t)^2}dx \notag\\
 & - \left(\frac{1}{2}{\delta}q_1-d_1\right)\int_0^1 {{e^{ - {\delta}x}}} \alpha {(x,t)^2}dx\notag\\
 &-\left(\frac{{r_c}}{2D}-q_2 {r_a}{e^{{\delta}}}c_0^2\right)\hat u {(1,t)}^2\notag\\&-\frac{{r_c}}{2D}\int_0^1 {{e^{{}x}}} \hat u {(x,t)^2}dx.\label{eq:dV1}
 \end{align}
{Recalling  conditions \eqref{eq:delta}--\eqref{eq:rc} on  ${\delta}$, $r_a$, $r_c$, and $r_d$, there exists a sufficiently small positive constant $\lambda_1$, such that}
\begin{align}
\dot V(t)\le -\lambda_1 V(t),\label{eq:dV2}
\end{align}
where
\begin{align}
{\lambda_1=\min\left\{\frac{{\lambda_{\rm min}(Q_1)}}{2\lambda_{\rm max}(P_1)},{\delta}q_2 -2d_4,{\delta}q_1-2d_1,\frac{1}{D}\right\}>0}.\label{eq:lam1}
\end{align}
Recalling \eqref{eq:barOmega}, we then obtain \eqref{eq:thm1}
{where the positive constant $\Upsilon$ is given as }
\begin{align}
\Upsilon=\frac{\xi_2\xi_4}{\xi_1\xi_3}.\label{eq:Upsilon}
\end{align}
The proof of the theorem is complete.
\end{proof}

Next, we will design a certainty equivalence delay-adaptive controller considering the nominal control action \eqref{eq:U} fed with an estimate $\hat D$,  which is given by an update law resulting from a triggered batch least-square identifier of the unknown delay $D$.
\section{ Delay-Adaptive controller}\label{sec:adaptive}
By replacing the unknown delay $D$ in the  nominal continuous-in-time feedback \eqref{eq:U} with the estimate $\hat D(t_i)$ that is generated with a triggered batch least-squares identifier that is designed later, we derive the following
delay-compensated adaptive  control law $U_{\rm d}$:
\begin{align}
 U_{\rm d} (t)= & \int_0^1 M_1(y;\hat D(t_i))z(y,t)dy+\int_0^1 M_2(y;\hat D(t_i))w(y,t)dy\notag\\&+\int_0^1 M_3(y;\hat D(t_i))  v(y,t)dy\notag\\&+M_4(\hat D(t_i))X(t), \quad t\in[t_i,t_{i+1}),\label{eq:dU}
\end{align}
where  $\{t_i\ge 0\}_{i=0}^{\infty}$, $i\in \mathbb Z^+$ is the sequence of time instants at which the delay estimate is updated. The triggering mechanism that determines the sequence of time instants is shown next.
\subsection{Triggering Mechanism}
The sequence of time instants $\{t_i\ge 0\}_{i=0}^{\infty}$, $i\in \mathbb Z^+$ ($t_0=0$), is
defined as
\begin{align}
&t_{i+1}=\begin{cases} \min\bigg\{\inf\big\{t>t_i:\Omega(t)=(1+a)\hat\Upsilon({\hat D(t_i)})\Omega(t_i)\big\},\\~~~~~~t_i+T\bigg\},~~~~for~\mbox{ $\Omega(t_i)\neq 0$}\\ t_i+T,~~~~~~~~~~~~for~\mbox{$\Omega(t_i)=0$} \end{cases}\label{eq:ri}
\end{align}
where the design parameters $a$, $T$ are positive and free, and the function $\Omega(t)$ is given in \eqref{eq:omeganorm}. {The function  ${\hat\Upsilon({\hat D(t_i)}})$ is  the overshoot coefficient
that is associated with the system transient and is obtained by replacing the unknown $D$ with $\hat D(t_i)$ in  $\Upsilon$ which is defined in \eqref{eq:Upsilon} (please note that $\xi_1$ and $\xi_2$  in \eqref{eq:Upsilon} depend on the delay $D$  through the delay-dependent kernel functions $K_1,K_2,\eta,R,P$ included in   \eqref{eq:xi1} and  \eqref{eq:xi2}. See Appendix C for further details). }

\subsection{Least-squares identifier for the unknown delay}\label{sec:ls}
Now, we design the identifier which stands as the update law of the estimated delay $\hat D$. According to   \eqref{eq:ode7},  for $\tau>0$ and $ n=1,2,\cdots$, the following equality holds:
\begin{align}
&\quad D\frac{d}{d\tau}\int_0^{1}\sin({x\pi n})v(x,\tau)dx={\pi n}\int_0^{1}\cos({x\pi n})v(x,\tau)dx.\label{eq:Ls1}
\end{align}
Integrating \eqref{eq:Ls1} from $0$ to $t$, yields
\begin{align}
D\int_0^{1}\sin({x\pi n})v(x,t)dx={\pi n}\int_{0}^{t}\int_0^{1}\cos({x\pi n})v(x,\tau)dxd\tau\label{eq:ab}
\end{align}
where \eqref{eq:ode9} has been recalled. Straightforwardly, \eqref{eq:ab} can be  written as
\begin{align}
&f_n(t)=Dg_{n}(t),\label{eq:f}
\end{align}
where
\begin{align}
&f_n(t)={\pi n}\int_{0}^{t}\int_0^{1}\cos({x\pi n})v(x,\tau)dxd\tau, \label{eq:fn}\\
&g_ {n}(t)=\int_0^{1}\sin({x\pi n})v(x,t)dx,\label{eq:gn}
\end{align}
for $n\in\mathbb N$. Define the function $h_{i,n}$ by the formula:
\begin{align}
h_{i,n}(\ell)=&\int_{\mu_{i+1}}^{t_{i+1}}(f_n(t)-\ell g_{n}(t))^2dt,~i\in  \mathbb Z^+, n\in\mathbb N, \label{eq:hi}
\end{align}
and  time instant $\mu_{i+1}$  as
\begin{align}
\mu_{i+1}:=\min\{{t_g}:g\in\{0,\ldots,i\},t_g\ge t_{i+1}-\tilde N T\},\label{eq:mui}
\end{align}
where the positive integer $\tilde N\ge 1$ is a free design parameter (in practice, a lager $\tilde N$ means a bigger set of data used in the least-squares
identifier, which makes the
identifier more robust with respect to measurement errors), and where the positive constant $T$ is the maximum dwell time according to \eqref{eq:ri}.
From \eqref{eq:f}, one can deduce that  the function $h_{i,n}(\ell)$ in \eqref{eq:hi} has a global minimum $h_{i,n}(D)=0$. Then, using Fermat's theorem (vanishing gradient at extrema),  the following matrix equation hold for every $i\in  \mathbb Z^+$ and $n\in\mathbb N$:
\begin{align}
H_{n}(\mu_{i+1},t_{i+1})=G_n(\mu_{i+1},t_{i+1})D\label{eq:Fer}
\end{align}
where
\begin{align}
H_{n}(\mu_{i+1},t_{i+1})&=\int_{\mu_{i+1}}^{t_{i+1}}g_{n}(t)f_{n}(t) dt,\label{eq:H1m}\\
G_{n}(\mu_{i+1},t_{i+1})&=\int_{\mu_{i+1}}^{t_{i+1}}g_{n}(t)^2 dt.\label{eq:Q1m}
\end{align}
Indeed, \eqref{eq:Fer} is obtained by differentiating the functions $h_{i,n}(\ell)$ defined by \eqref{eq:hi} with respect to $\ell$, and evaluating the derivative (zero) at the global minimum $\ell=D$. Using \eqref{eq:Fer}--\eqref{eq:Q1m}, the following delay identifier is constructed:
\begin{align}
&\hat D(t_{i+1})={\rm argmin}\bigg\{|\ell-\hat D(t_i)|^2: {\underline D\le \ell\le \overline D}, \notag\\
&H_{n}(\mu_{i+1},t_{i+1})=G_n(\mu_{i+1},t_{i+1})\ell,~~ n=1,2,\cdots\bigg\},~~i\in\mathbb Z^+.\label{eq:adaptivelaw}
\end{align}
\begin{remark}[Implementation of the identifier] Implementation of the identifier begins with calculating $H_{n}(\mu_{i+1},t_{i+1})$, $G_n(\mu_{i+1},t_{i+1})$ from $n=1$, $i=0$, i.e., $H_1(\mu_{1}, t_{1})$, $G_1(\mu_{1}, t_{1})$, using \eqref{eq:H1m}, \eqref{eq:Q1m}, \eqref{eq:fn}, \eqref{eq:gn}. If $G_1(\mu_{1}, t_{1})\neq0$, it implies that $\ell$ belongs to a singleton set, i.e.,  $\ell=\frac{H_{1}(\mu_{1},t_{1})}{G_1(\mu_{1},t_{1})}$. It is followed that the output of the identifer \eqref{eq:adaptivelaw} at $t_1$ is $\hat D(t_{1})=\frac{H_{1}(\mu_{1},t_{1})}{G_1(\mu_{1},t_{1})}$. If $G_1(\mu_{1}, t_{1})=0$, we continue  to calculate $H, G$ with $n=2$, $i=0$, i.e., $H_2(\mu_{1}, t_{1})$, $G_2(\mu_{1}, t_{1})$, and then evaluate the value of $G_2(\mu_{1}, t_{1})$. Similarly, if  $G_2(\mu_{1}, t_{1})\neq0$, the output of the identifier at $t_1$  is $\hat D(t_1)=\frac{H_{2}(\mu_{1},t_{1})}{G_2(\mu_{1},t_{1})}$. If $G_2(\mu_{1}, t_{1})=0$, then move to calculate the case of $n=3$, $i=0$, i.e., $H_3(\mu_{1}, t_{1})$, $G_3(\mu_{1}, t_{1})$. Repeating the above steps, until we find a $G_n(\mu_{1}, t_{1})\neq 0$ for a certain $n$,  the output of the identifier at $t_1$  is  $\hat D(t_1)=H_n(\mu_{1}, t_{1})/G_n(\mu_{1}, t_{1})$. For saving the computation time, we can set an upper limit $\bar n$ for $n$. That is, if $G_n(\mu_{1}, t_{1})=0$  for all $n=1,\cdots,\bar n$, we then stop the seeking at the updating time $t_1$ and consider $\ell$ belongs to the original set $\{\ell\in\mathbb R:\underline D\le \ell\le \overline D\}$ which leads to the output of the identifier is equal to the estimate at the last time instant, i.e., $\hat D(t_1)=\hat D(t_0)$, according to \eqref{eq:adaptivelaw}. The same computation process is followed for the subsequent updating time instants $t_2,t_3,\ldots$. For  many practical applications, such as simulating a deep-sea construction vessel, locating non-zero values of $G_n(\mu_{i+1}, t_{i+1})$ is a straightforward task following the algorithm described above.  \label{rm:impid}
\end{remark}

{Please note that even though the actuator states $v(x,t)$ are measurable in this full-state feedback case, for the delay estimation, one cannot adopt the ``naive" method--that is, taking the time and spatial derivatives of the signal $v(x,t)$  to calculate $d$ in \eqref{eq:ode7} straightforwardly-- because of the following two reasons:} 1) taking the time derivative of the measured signals always leads to the undesired noise amplification in practice; 2) the possible zero values of $v_t(x,t)$ accompanied with the unknown delay $D$ will {engender} singularity.
{\begin{prope}[Existence of solution in an interval]\label{Pro:1}
For every $(z[t_i],w[t_i],v[t_i])^T\in L^2((0,1);\mathbb R^3)$, $X(t_i)\in \mathbb R^m$,  there exists  a unique (weak) solution $((z,w,v)^T,X)\in C^0([t_i,t_{i+1}];L^2(0,1);\mathbb R^3)\times C^0([t_i,t_{i+1}];\mathbb R^m)$ to  the system \eqref{eq:ode1}--\eqref{eq:ode4}, \eqref{eq:ode6}--\eqref{eq:ode9}, \eqref{eq:dU}.
\end{prope}}

\begin{proof} Similar to the proof of Proposition \ref{Pro:1} in \cite{JiAdaptive2021}, applying three transformations to decouple PDEs and ODE, and recalling the result in the part 1 in Appendix of \cite{Davo2018Stability}, Proposition \ref{Pro:1} can be obtained. {We direct the reader} to the proof of Proposition \ref{Pro:1} in \cite{JiAdaptive2021} for details.
\end{proof}

\section{Main result}\label{sec:sta}
{Before presenting the main theorem, we propose the following technical lemmas,} where when we say that $v(x,t)$ is equal to zero for $x\in[0,1],t\in[\mu_{i+1},t_{i+1}]$, or not identically zero on the same domain, we mean except possibly for finitely many discontinuities of the functions $v(x,t)$. These discontinuities are isolated curves in the rectangle $[0,1]\times[\mu_{i+1},t_{i+1}]$.

\begin{lema}[$G_{n}(\mu_{i+1},t_{i+1})=0$]\label{cl:nsQ0}
The sufficient and necessary condition of $G_{n}(\mu_{i+1},t_{i+1})=0$ for all $n\in\mathbb N$ is $v[t]= 0$ on $t\in[\mu_{i+1},t_{i+1}]$.

\end{lema}
\begin{proof}
Necessity: If $G_{n}(\mu_{i+1},t_{i+1})=0$ for all $n\in\mathbb N$, then the definition \eqref{eq:Q1m} in conjunction with continuity of $g_{n}(t)$ for $t\in[\mu_{i+1},t_{i+1}]$ (because of the definition \eqref{eq:gn} and the fact that $v\in C^0([t_i,t_{i+1}];L^2(0,1))$ in Proposition \ref{Pro:1}) implies
\begin{align}
g_{n}(t)=0,~~ t\in[\mu_{i+1},t_{i+1}].\label{eq:gq10}
\end{align}
According to the definition \eqref{eq:gn}, the equation \eqref{eq:gq10} implies
\begin{align}
\int_0^{1}\sin({x\pi n})v(x,t)dx=0, ~~t\in[\mu_{i+1},t_{i+1}]
\end{align}
for all $n\in\mathbb N$.
Since the set $\{\sqrt{2}\sin(x\pi n):n=1,2,\ldots\}$ is an orthonormal basis of $L^2(0,1)$, we have $v[t]=0$ for $t\in[\mu_{i+1},t_{i+1}]$.

Sufficiency: If $v[t]=0$ on $t\in[\mu_{i+1},t_{i+1}]$, then $G_{n}(\mu_{i+1},t_{i+1})=0$ for all $n\in\mathbb N$ is obtained directly by recalling  \eqref{eq:Q1m} and \eqref{eq:gn}.
\end{proof}
\begin{lema}[The identifier properties at $t_{i+1}$]\label{lem:theta}
For the adaptive estimates defined by \eqref{eq:adaptivelaw}, the following statements hold:
\begin{itemize}
\item If $v[t]$ is not identically zero for $t\in[\mu_{i+1},t_{i+1}]$, then $\hat D(t_{i+1})=D$.
\item If $v[t]$ is  identically zero for $t\in[\mu_{i+1},t_{i+1}]$, then $\hat D(t_{i+1})=\hat D(t_{i})$.
\end{itemize}
\end{lema}
\begin{proof}
Define a set
\begin{align}
S_i=\bigg\{\underline D\le \ell\le \overline D:~ &H_{n}(\mu_{i+1},t_{i+1})=G_n(\mu_{i+1},t_{i+1})\ell,\notag\\& n=1,2,\cdots\bigg\}.\label{eq:Si}
\end{align}
From \eqref{eq:Fer}, we know that $D\in S_i$. If $S_i$ is a singleton, it is nothing else but the generated adaptive estimate $\hat D(t_{i+1})$ by \eqref{eq:adaptivelaw}, which is equal to the true delay $D$.

\begin{enumerate}
\item If $v[t]$ is not identically zero for $t\in[\mu_{i+1},t_{i+1}]$, recalling Lemma \ref{cl:nsQ0}, there exists $n\in{\mathbb N}$ such that $G_{n}(\mu_{i+1},t_{i+1})\neq0$. Now defining the index set $I$ as the set of all $n\in{\mathbb N}$ with $G_{n}(\mu_{i+1},t_{i+1})\neq0$, then \eqref{eq:Si} implies that $$S_i=\left\{\ell=\frac{H_{n}(\mu_{i+1},t_{i+1})}{G_{n}(\mu_{i+1},t_{i+1})}, n\in I\right\}$$
is a singleton, and therefore from  \eqref{eq:adaptivelaw} we get  $\hat D(t_{i+1})=D$.

\item  If $v[t]$ is identically zero on $t\in[\mu_{i+1},t_{i+1}]$, according to \eqref{eq:fn}, \eqref{eq:gn}, \eqref{eq:H1m}, \eqref{eq:Q1m}, one obtains $$G_{n}(\mu_{i+1},t_{i+1})=H_{n}(\mu_{i+1},t_{i+1})=0,~~n\in \mathbb N,$$
and it  follows that {$S_i=\{\underline D\le \ell\le \overline D\}$}. Then, from  \eqref{eq:adaptivelaw} one arrive at  $\hat D(t_{i+1})=\hat D(t_{i})$.
\end{enumerate}
The proof is complete.
\end{proof}
\begin{lema}[The identifier properties for $t\in[t_i,\lim_{k\to \infty}(t_k))$]\label{lem:keep}
If $\hat D(t_{i})=D$ for certain $i\in  \mathbb Z^+$, then $\hat D(t)=D$ for all $t\in[t_i,\lim_{k\to \infty}(t_k))$.
\end{lema}
\begin{proof}
According to Lemma \ref{lem:theta}, we have that $\hat D(t_{i+1})$ is equal to either $D$ or $\hat D(t_i)$. Therefore, if  $\hat D(t_i)=D$, then $\hat D(t_{i+1})=D$. Repeating this process, we then have $\hat D(t)=D$ for all $t\in[t_i,\lim_{k\to \infty}(t_k))$.  The proof is complete.
\end{proof}
\begin{lema}[Existence of a minimum dwell-time]
There exists a positive constant $\tau_d$ such that $t_{i+1}-t_i\ge \tau_d$ for all $i\in \mathbb Z^+$.\label{lem:taud}
\end{lema}
\begin{proof}
The result is established by discussing the following two cases:

\begin{itemize}
\item {Case 1: The finite-time exact identification is not achieved.}  According to Lemmas \ref{lem:theta} and \ref{lem:keep}, we have $\hat D(t)=\hat D(0)$ in the control law. It implies that there is no discontinuity in the control input, and thus $\Omega(t)$ defined in \eqref{eq:omeganorm} is continuous and differentiable. Recalling the triggering mechanism \eqref{eq:ri}, the lower bound $\underline\tau_i$ of the dwell time is given by \begin{align}
\underline\tau_i=\begin{cases} \min\bigg\{\frac{((1+a){\hat\Upsilon({\hat D(0)}})-1)\Omega(t_i)}{\max_{t\in(t_i,t_{i+1})} |\dot \Omega(t)|},T\bigg\}>0,&\mbox{if $\Omega(t_i)\neq 0$}\\ T,&\mbox{if $\Omega(t_i)=0$}. \end{cases}\label{eq:taud0}
\end{align}
Therefore,  the lower bound of the minimal dwell time is written as
\begin{align}
\tau_{d}=\min\{\underline\tau_i:i\in \mathbb Z^+\}>0. \label{eq:taud1}
\end{align}

\item {Case 2: The finite-time exact identification is achieved.}

Before the time instant when the estimate $\hat D(t)$  reaches the actual value $D$, e.g., $t_{f}$, we know $\hat D(t)=\hat D(0)$ for $t<t_{f}$ by applying Lemma \ref{lem:theta} repeatedly. Through the same analysis in case 1, we have the lower bound of the minimal dwell time before $t_{f}$ is $\tau_{d}$ given by \eqref{eq:taud0}, \eqref{eq:taud1}.

After the time instant $t_{f}$,
the exact identification is achieved and   $t_{i+1}-t_i=T\ge \tau_d$ for all $i\ge f$. We prove this as follows. Once the exact delay identification is achieved, the delay-adaptive control input is identical to the nominal delay-compensated control input in Section \ref{sec:nominalcontrol}. When $\Omega(t_i)\neq0$, we have that $\Omega(t)\le\Upsilon \Omega(t_i)$ for $t_i\le t\le t_{i+1}$, $i\ge f$ according to Theorem \ref{theorem:1}. It follows from $\hat\Upsilon({\hat D(t_i)})=\hat\Upsilon(D)=\Upsilon$ for all $i\ge f$ that  $\Omega(t)< (1+a)\hat\Upsilon({\hat D(t_i)}) \Omega(t_i)$ for $t_i\le t\le t_{i+1}$, $i\ge f$. Thus $t_{i+1}-t_i=T$ according to \eqref{eq:ri}. When $\Omega(t_i)=0$, we straightforwardly have $t_{i+1}-t_i=T$ according to the second equation in \eqref{eq:ri} and therefore, $t_{i+1}-t_i=T$ for all $i\ge f$.
\end{itemize}
We finally conclude that the lower bound of the minimal dwell time in case 2 is $\tau _d$ as well.
\end{proof}
\begin{coro}[Well-posedness of the closed-loop system]\label{col:1}
No Zeno phenomenon occurs, i.e.,
$\lim_{i\to \infty} t_i=+\infty$,
and the closed-loop system is well-posed in the sense that for every $(z[0],w[0])^T\in L^2((0,1);\mathbb R^2)$, {$X(0)\in \mathbb R^m$}, and $\hat D(0)\in [\underline D,\overline D]$,  there exists a unique (weak) solution $((z,w,v)^T,{X})\in C^0(\mathbb R_+;L^2(0,1);\mathbb R^3)\times C^0(\mathbb R_+;\mathbb R^m)$, and $\hat D(t)\in \{\ell\in\mathbb R: \underline D\le \ell \le\overline D\}$ for $t\in[0,\infty)$, to the system consisting of \eqref{eq:ode1}--\eqref{eq:ode4}, \eqref{eq:ode6}--\eqref{eq:ode9}, \eqref{eq:dU}, and  \eqref{eq:adaptivelaw}.
\end{coro}
\begin{proof}
Recalling Lemma \ref{lem:taud}, we have that
\begin{align*}
t_i\ge \tau_d i,  ~~~i\in \mathbb Z^+,
\end{align*}
where $\tau_d>0$, that is,
\begin{align}
\lim_{i\to \infty}(t_i)=+\infty,\label{eq:tiinfty}
\end{align}
which implies a  solution  defined on $\mathbb R_+$ in the subsequent analysis.

From the initial data $(z[0],w[0])^T\in L^2((0,1);\mathbb R^2)$, {$X(0)\in \mathbb R^m$} and \eqref{eq:ode9}, recalling the result in Proposition \ref{Pro:1} for $i=0$, it {follows} that $((z,w,v)^T,X)\in C^0([t_0,t_{1}];L^2(0,1);\mathbb R^3)\times C^0([t_0,t_{1}];\mathbb R^m)$, {which implies $(z[t_1],w[t_1],v[t_1])^T\in L^2((0,1);\mathbb R^3)$, {$X(t_1)\in \mathbb R^m$}.} Recalling the result in Proposition \ref{Pro:1} for $i=1$, together with the solution obtained for $[t_0,t_{1}]$, we have that $((z,w,v)^T,X)\in C^0([t_0,t_{2}];L^2(0,1);\mathbb R^3)\times C^0([t_0,t_{2}];\mathbb R^m)$. Repeating the above steps, we obtain that $((z,w,v)^T,X)\in C^0([t_{0},t_{i}];L^2(0,1);\mathbb R^3)\times C^0([t_{0},t_{i}];\mathbb R^m)$ for $i\in \mathbb N$. Applying \eqref{eq:tiinfty}, we thus have $((z,w,v)^T,X)\in C^0(\mathbb R_+;L^2(0,1);\mathbb R^3)\times C^0(\mathbb R_+;\mathbb R^m)$. It is straightforwardly obtained from \eqref{eq:adaptivelaw} that $\hat D(t)\in [\underline D,\overline D]$ if $\hat D(0)\in [\underline D,\overline D]$.

Corollary \ref{col:1} is thus obtained.
\end{proof}
\begin{lema}[Finite-time convergence of the update law]\label{lem:convergence}
If the control input $U_{\rm d}(t)$ is not identically zero on $t\in[0,\infty)$, the estimate $\hat D$ converges to the true value in finite time, i.e.,
\begin{align}
&\hat D(t)=D, ~\forall t\in[ t_{f},\infty)\label{eq:convergence}\\and ~~&\hat D(t)=\hat D(0),~\forall t\in[0,t_{f})~(for~f>0) \label{eq:convergence1}
\end{align}
where $$t_{f}=\min\{t_i:i\in\mathbb N,t_i> \min\left\{t\ge0: U_{\rm d}(t)\neq 0\right\}\}.$$  If the control input $U_{\rm d}(t)$ is identically zero on $t\in[0,\infty)$, $\hat D(t)\equiv \hat D(0)$ all the time.
\end{lema}
\begin{proof}
If the control input $U_{\rm d}(t)$ is not identically zero on $t\in[0,\infty)$, there exists a time instant $t_f$ $(f>0)$ such that $U_{\rm d}(t)$ is not identically zero in  $t\in[\mu_{f},t_{f}]$. Recalling \eqref{eq:ode7}, \eqref{eq:ode8} where $U(t)$ has been replaced by $U_{\rm d}(t)$, we conclude that the actuator  state $v[t]$ is not identically zero on $t\in[\mu_{f},t_{f}]$. Recalling Lemma  \ref{lem:theta}, Lemma \ref{lem:keep}, and Corollary \ref{col:1}, we thus obtain $\hat D(t)=D, ~\forall t\in[ t_{f},\infty)$ \eqref{eq:convergence}.

For $t\le t_{f-1}$ $(f>1)$, $U_{\rm d}(t)=0$, i.e., $v[t]\equiv 0$, it follows from Lemma  \ref{lem:theta} that $\hat D=\hat D(0)$ on $t\in[0, t_{f})$. If $f=1$, we directly have $\hat D=\hat D(0)$ on $t\in[0, t_{f})$ because the identifier has not been updated. We thus obtain \eqref{eq:convergence1}.

If $U_{\rm d}(t)$ is identically zero on $t\in[0,\infty)$, we know that $v[t]\equiv 0$ all the time. Recalling Lemma  \ref{lem:theta} and Corollary \ref{col:1}, we have $\hat D(t)\equiv \hat D(0)$ on $t\in[0,\infty)$.
\end{proof}

By virtue of Lemma \ref{lem:convergence}, a sufficient and necessary condition of exact identification of the unknown delay is the fact that $U_{\rm d}(t;\hat D(0))$ is not identically zero on $t\in[0,\infty)$.
Next, we show that we always can find $\hat D(0)$ such that $U_{\rm d}(t)$ is not identically zero on $t\in[0,\infty)$ if the exponential stability \eqref{eq:thm1} is not ensured in the open loop (if \eqref{eq:thm1} is already achieved in the open loop, the control design would be trivial).
Define a set $\mathcal D$ of all $\hat D(0)\in[\underline D,\overline D]$ such that $U_{\rm d}(t;\hat D(0))$ is not identically zero on $t\in[0,\infty)$. We show that $\mathcal D$ is a nonempty set in the following lemma.

\begin{lema}[Nonempty set $\mathcal D$]\label{lem:D}
If the considered plant \eqref{eq:ode1}--\eqref{eq:ode5} in the open loop, i.e., with the identically zero control input $U$,  is not an exponentially stable system, the set $\mathcal D$ of all $\hat D(0)\in[\underline D,\overline D]$ such that $U_{\rm d}(t;\hat D(0))$ is not identically zero on $t\in[0,\infty)$ is non-empty.
\end{lema}
\begin{proof}
Suppose that $\mathcal D$ is {an empty set}, equivalently,  $U_{\rm d}(t;\hat D(0))$ is identically zero on $t\in[0,\infty)$ for all $\hat D(0)\in[\underline D,\overline D]$. It implies that $U_{\rm d}(t;D)$, i.e., the nominal control input $U(t)$ in Theorem \ref{Pro:1} that shows the exponential stability of the nominal closed-loop system, is identically zero on $t\in[0,\infty)$. {Thus,} the plant with the identically zero control input (the open-loop system) has already been exponentially stable, which contradicts with that the considered plant in open loop is not an exponentially stable system. Therefore, Lemma \ref{lem:D} is obtained.
\end{proof}
\begin{remark}[{On the choice of   $\hat D(0)$ in practice}]\label{rm:D0}
{ {When $\Omega(0)\neq 0$, a simple way to set the initial condition $\hat D(0)$ is to find a $\hat D(0)$ such that $U_{\rm d}(0;\hat D(0))$ defined in \eqref{eq:dU} is nonzero, which is available in most practical applications.} As a result, the exact parameter identification would be achieved at the first triggering time. If $\Omega(0)= 0$ that means that the initial values of all plant states are zero, {and consequently,  the states of the plant \eqref{eq:ode1}--\eqref{eq:ode5} are identically zero all the time under the zero control input. Therefore, the control input vanishes with no need to construct an estimate of the delay, namely, $\hat D$.}}
\end{remark}

Now, we are in a position to state our main result in the following theorem, i.e., exponential regulation of the plant and actuator states.
\begin{remark}\label{rem1}
If the exponential stability has already been achieved in open loop, the control design would be trivial, because the control input can just be set as zero even though the unknown delay exists. We only discuss the case that the open-loop system is not exponentially stable in the following theorem.
\end{remark}
\begin{thme}\label{th:part1}
For all initial data $(z[0],w[0])^T\in L^2(0,1)$, $X(0)\in \mathbb R^m$, $\hat D(0)\in\mathcal D$,  considering the closed-loop system consisting of the plant \eqref{eq:ode1}--\eqref{eq:ode4}, \eqref{eq:ode6}--\eqref{eq:ode9}, the controller \eqref{eq:dU}, the triggering mechanism \eqref{eq:ri}, and the least-squares identifier  \eqref{eq:adaptivelaw}, the exponential regulation of the closed-loop system is obtained in the sense that there exist positive constants $M,\lambda_1$ such that
\begin{align}
\Omega(t)\le M\Omega(0)e^{-\lambda_1 t},~~t\ge 0,\label{eq:th1norm1}
\end{align}
where $\Omega(t)$ is defined in \eqref{eq:omeganorm}.
\end{thme}

\begin{proof}
{Replacing the nominal control law $U$  by the  delay-adaptive control law $U_d$ defined by \eqref{eq:dU} in \eqref{eq:ode8},  through the transformations in Section \ref{sec:nominalcontrol}, the right boundary condition of the actuator PDE \eqref{eq:tar14} in the target system \eqref{eq:targ5}--\eqref{eq:targ4}, \eqref{eq:tar12}--\eqref{eq:tar14} becomes}
\begin{align}
\hat u(0,t)=\xi(t),\label{eq:rbtar}
\end{align}
where
\begin{align}
\xi(t)=&-\int_0^1 (M_1(y;D)-M_1(y;\hat D))z(y,t)dy\notag\\
&-\int_0^1 (M_2(y;D)-M_2(y;\hat D))w(y,t)dy\notag\\&+\int_0^1 (M_3(y;D)-M_3(y;\hat D))  v(y,t)dy\notag\\&+(M_4(D)-M_4(\hat D)) {X}(t).\label{eq:xi}
\end{align}
{Taking the derivative of \eqref{eq:V} along the target system states corresponding to the even-based closed-loop system consisting of \eqref{eq:targ5}--\eqref{eq:targ4}, \eqref{eq:tar12}, \eqref{eq:tar13}, and \eqref{eq:rbtar},
through a similar process in \eqref{eq:dV1}, recalling  conditions \eqref{eq:delta}--\eqref{eq:rd} on  ${\delta}$, $r_a$, $r_c$, and $r_d$  (we emphasize that conditions \eqref{eq:delta}--\eqref{eq:rd} only depend on  the known plant parameters and the known bounds of the unknown parameters in Assumption \ref{as:coe1}), we obtain}
\begin{align}
\dot V(t)\le -\lambda_1 V(t)+\frac{{r_c}}{2D}\xi(t)^2,~~t\ge 0,\label{eq:dVa0}
\end{align}
where
$\lambda_1$ is given in \eqref{eq:lam1}.  From Lemma \ref{lem:convergence}, \ref{lem:D} and Remark \ref{rem1}, the finite-time convergence of the parameter estimate $\hat D$ to the actual value $D$ is obtained directly.
Now, using  \eqref{eq:convergence} and \eqref{eq:xi}, one can establish that
\begin{align}
\xi(t)\equiv 0,~~t\in[t_{f},\infty).\label{eq:y0}
\end{align}
We then have that
\begin{align}
\dot V(t)\le -\lambda_1 V(t),~~t\ge t_{f}.\label{eq:dVa1}
\end{align}
{Multiplying   both sides of \eqref{eq:dVa1} by $e^{\lambda_1t}$ and integrating the resulting terms from $t_{f}$ to $t$ lead to the following inequality}
$$V(t)\le V(t_{f})e^{-\lambda_1 (t-t_{f})},~~t\ge t_{f},$$ {which, by virtue of  \eqref{eq:barOmega}, is equivalent to}
 \begin{align}
\Omega(t)\le \Upsilon\Omega(t_{f})e^{-\lambda_1 (t-t_{f})},~~ t\ge t_{f},\label{eq:Vnorm1}
\end{align}
{where $\Omega$ is defined in   \eqref{eq:omeganorm} and the positive constant $\Upsilon$ is given in \eqref{eq:Upsilon} }

{Note that the norm estimate \eqref{eq:Vnorm1}  is only true for $t\ge t_{f}$  and if $t_{f}=0$, \eqref{eq:th1norm1} is obtained directly. }

{Next, we extend our analysis for $t\in[0,t_{f}]$ with $t_{f}\neq 0$. With the help of  \eqref{eq:barOmega}, \eqref{eq:xi}, we obtain from \eqref{eq:dVa0} that}
\begin{align}
\dot V(t)\le& - \lambda_1 V(t)+ Q(\hat D(0))V(t),~~~t\in[0,t_{f}],\label{eq:dvcal2b}
 \end{align}
where the positive constant $Q(\hat D(0))$  is
\begin{align}
Q(\hat D(0))=&\max\bigg\{\max_{i=1,2,3;y\in[0,1]}\left\{(M_i(y;D)-M_i(y;\hat D(0)))^2\right\},\notag\\& (M_4(D)-M_4(\hat D(0)))^2\bigg\}\frac{2{r_c}}{D\xi_1\xi_3}
\end{align}
which is derived  by finding an {upper bound} for  $\xi(t)^2$ given in \eqref{eq:xi} and recalling \eqref{eq:barOmega}.

Hence, the following holds
    \begin{align}
\Omega(t)\le \Upsilon\Omega(0) e^{\lambda_2(\hat D(0))t},~~t\in[0,t_{f}],\label{eq:dvcal2c}
 \end{align}
where $$\lambda_2(\hat D(0))=|Q(\hat D(0))-\lambda_1|>0,$$  and the positive constant $\Upsilon$ is given in \eqref{eq:Upsilon}.
Therefore, it straightforwardly follows that
\begin{align}
\Omega(t_{f})\le \Upsilon e^{\lambda_2(\hat D(0))t_{f}}\Omega(0).\label{eq:normbound}
\end{align}
Considering \eqref{eq:dvcal2c}, combining \eqref{eq:Vnorm1} and   \eqref{eq:normbound}  yields
 \begin{align}
\Omega(t)\le \Upsilon^2 e^{(\lambda_2(\hat D(0))+\lambda_1)t_{f}}\Omega(0)e^{-\lambda_1 t},~~~~t\ge 0,\label{eq:Vnorm2}
\end{align}
which is equivalent to  \eqref{eq:th1norm1} with
$$M=\Upsilon^2 e^{(\lambda_2(\hat D(0))+\lambda_1)t_{f}}.$$ The proof of the theorem is complete.
\end{proof}
\section{Simulation}\label{sec:sim}
A deep-sea construction vessel (DCV) is used to place equipment to be installed at the predetermined location on the seafloor, which is shown in Figure \ref{fig:DCV}.
Different from \cite{J2022Delay-compensated} that deals with a known sensor delay that exists in the large-distance transmission of the sensing signal from the seafloor to the
vessel on the ocean surface through a set of acoustics devices, we consider all possible delays here (including the transmission of the sensing signal, computation of the control law, and the delay in the hydraulic actuator for the ship-mounted crane and so on) as an unknown delay in the control input channel.  By designing a control input at the top of the crane cable, our goal is to reduce the oscillations of the crane cable with the purpose of placing the payload attached at the bottom of the cable in the target area, despite the presence of the unknown delay.
\begin{figure}
\centering
\includegraphics[width=6cm]{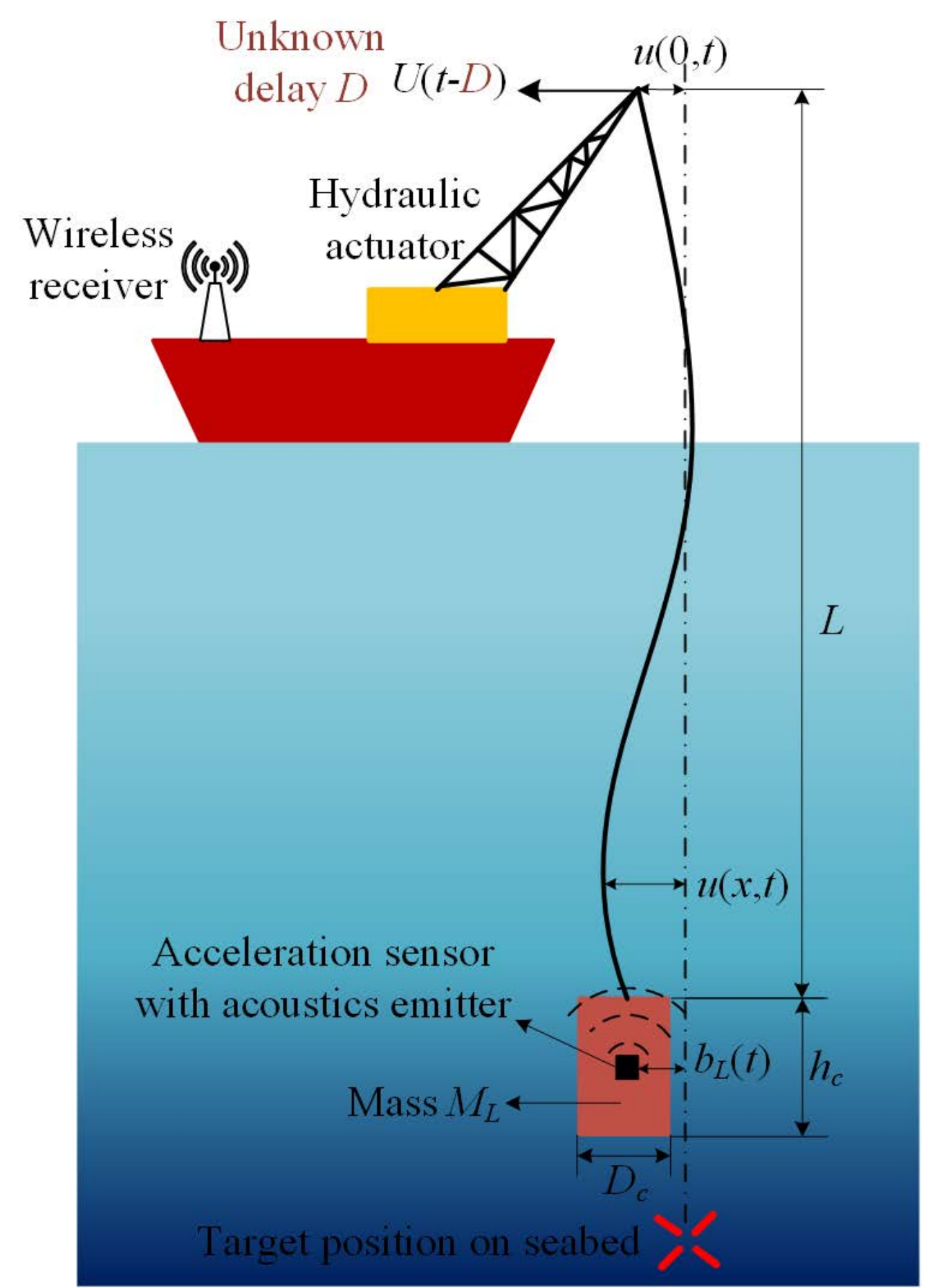}
\caption{Deep-sea construction vessel.}
\label{fig:DCV}
\end{figure}
\subsection{Model}\label{sec:simmodel}
The following dynamic model of cable-payload lateral oscillations in DCV is taken from \cite{J2022Delay-compensated},
\begin{align}
T_0u_{\bar x}(0,t) &= U(t-D),\label{eq:5.2}\\
\rho u_{tt}(\bar x,t) &= T_0u_{\bar x\bar x}(\bar x,t) - d_c{u_t}(\bar x,t),\label{eq:5.3}\\
u(L,t) &= b_L(t),\label{eq:5.4}\\
M_L\ddot b_L(t) &= -d_L\dot b_L(t) + T_0u_{\bar x}(L,t),\label{eq:5.5}
\end{align}
$\forall (\bar x,t) \in [0,L]\times[0,\infty)$.  {The state $u(\bar x,t)$ describes the lateral oscillation displacement  along the cable, and $b_L(t)$ denotes that of the payload. The control input $U$ is subject to the unknown time delay $D$ mentioned above.
The static tension $T_0$  is defined as
$$T_0=M_Lg-F_{\rm buoyant},$$
where the buoyancy $F_{\rm buoyant}$ is
$$F_{\rm buoyant}=\frac{1}{4}\pi D_c^2 h_c\rho_{s}g.$$
The physical parameters of the deep-sea construction vessel are shown in Table \ref{table1}.

Like \cite{J2022Delay-compensated}, after applying the Riemann transformations
\begin{align}
{z}(\bar x,t) &= {u_t}(\bar x,t) - \sqrt {\frac{T_0}{\rho}} {u_{\bar x}}(\bar x,t),\label{eq:t11}\\
{w}(\bar x,t) &= {u_t}(\bar x,t) + \sqrt {\frac{T_0}{\rho}} {u_{\bar x}}(\bar x,t),\label{eq:t12}
\end{align}introducing a  space normalization variable
\begin{align}x=\frac{\bar x}{L}\in[0,1]\label{eq:barx},\end{align}
and defining $$X(t)=\dot b_L(t),$$  equations \eqref{eq:5.2}--\eqref{eq:5.5} are rewritten  as the considered plant \eqref{eq:ode1}--\eqref{eq:ode5} with the coefficients
\begin{align}
 &c_0=2\sqrt{\frac{1}{T_0\rho}},~~{q_1}=q_2=\frac{1}{L}\sqrt {\frac{T_0}{\rho}},\label{eq:simpa1}\\&d_1=d_2=d_3=d_4=\frac{{{-d_c}}}{2\rho},q=-1,~p=1,\\
 &C=2,~{A}=\frac{-d_L}{M_L}+\frac{\sqrt{T_0\rho}}{M_L},~{B}=-\frac{\sqrt {{T_0}\rho}}{M_L},\label{eq:simpa3}
\end{align}
which is the simulation model in this section, where it can be checked that the plan parameters in \eqref{eq:simpa1}--\eqref{eq:simpa3} satisfy Assumptions \ref{as:AB}, \ref{as:pq} by recalling Table \ref{table1}.
\begin{table}[!t]
\centering
\caption{Physical parameters of the DCV.}
\begin{tabular}{lccc}
\hline
Parameters (units)&values\\ \hline
Cable length ${L}$ (m) &1500\\
Cable linear density ${\rho}$ (kg/m) &7.5\\
Payload mass ${M_L}$ (kg) &3.5$\times$ $10^5$\\
Gravitational acceleration $g$ (m/s$^{2}$) &9.8\\
Cable material damping coefficient $d_c$ (N$\cdot$s/m) & 0.8\\
Height of payload modeled as a cylinder $h_c$ (m)& 7.5\\
Diameter of payload modeled as a cylinder $D_c$ (m)& 5\\
Damping coefficient at payload $d_L$ (N$\cdot$s/m) & 1.2$\times$$10^5$\\
Seawater density $\rho_s$ (kgm$^{-3}$)& $1024$  \\\hline
\end{tabular}
\label{table1}
\end{table}

The initial conditions are defined as $$z(x,0)=8\sin( {5\pi x}(1-x)),~~w(x,0)=-8\cos( 5{\pi x}),$$ thereby, $$X(0)=1.13,$$ recalling \eqref{eq:ode4}, which physically corresponds to the initial oscillation velocities of the payload.  The unknown delay $D$ is set as 1, and the known bounds $\underline D$ and $\overline D$ are assumed as 0.01 and 2. We will show the simulation results  of the following four cases:
\begin{itemize}
\item Open loop: the control input is zero;

\item Nonadaptive control: the nominal delay-compensated control with the unknown delay $D$ replaced by its estimate $0.25$;

\item Delay-adaptive control with the initial delay estimate $\hat D(0)=0.25$, where the design parameter $K$ in \eqref{eq:Am} is chosen as  $K=-18$;

\item Delay-adaptive control with the initial delay estimate $\hat D(0)=1.5$, where the design parameter $K$ in \eqref{eq:Am} is chosen as  $K=-13$;
\end{itemize}
Other design parameters are $$\delta=-0.36,~r_a=1.02,~r_c=1,$$$$r_d=0.02,~a=2,~T=3.12,~\tilde N=10,$$ according to \eqref{eq:delta}--\eqref{eq:rd}, where the last three parameters are free but positive. The parameter $\bar n$ mentioned in Remark \ref{rm:impid} is set as $\bar n=2$.
\begin{remark}\label{rm:sim}
In addition to Remark \ref{rm:impid} and Remark \ref{rm:D0} about the implementation of the delay identifier, some more things are worth noting in the simulation. 1) Approximating the integration with respect to the space variable in the identifier as the summation operator will cause a tiny error between the final parameter estimate and the true value in the simulation result, which will be seen in Fig. \ref{fig:estimate}. The smaller space step adopted in the simulation will make the error smaller. 2) The error of approximation in the simulation will also lead to tiny differences between the outputs of the identifer at each updating time even if the effective parameter deification has been achieved. Therefore, we set a small margin to tolerate the approximation error, that is--if the difference between the estimates from the identifier at two adjacent updating times is smaller than $2\%$ of the true value, we consider that this difference is caused by the approximation error in the simulation, and thus keep the estimate value as same as the one at the former updating time.
\end{remark}
\subsection{Simulation result}
The numerical computation is conducted using the finite difference method with the step sizes of $t$ and $x$ as 0.001, and 0.02, respectively. The approximate solutions of the kernel PDEs used in the control law, which is defined by \eqref{eq:dU}, \eqref{eq:ri}, \eqref{eq:adaptivelaw} where the integral operators
are approximated by sums, are also solved by the finite difference method based on the discretization of the triangular domain into a uniformly spaced grid with the interval of $0.02$.

The designed delay-adaptive control input and the estimate of the unknown delay are shown in Figures \ref{fig:control} and \ref{fig:estimate}, respectively, from which we know that the identification of the unknown delay is achieved at the first triggering time, no matter the initial delay estimate is less than ($\hat D(0)=0.25$) or larger than ($\hat D(0)=1.5$) the true value $D=1$. As mentioned in Remark \ref{rm:sim}, the tiny differences between the delay estimate and its true values come from the error of approximation---that is, approximating the integration with respect to the space variable from 0 to 1 in the identifier as the summation operator for the 51 spatial discrete points with the fixed interval of $0.02$. The time evolution of the ODE state $X(t)$ is shown in Figure \ref{fig:X}, where the brown dashed line, the red dashed line, the black solid line,  and the blue dot-dash line show the results of the four cases mentioned in Section \ref{sec:simmodel}, respectively. Although both the nonadaptive delay-compensated controller and the delay-adaptive controllers can attenuate the state of the ODE in comparison to the open loop scenario, Figure \ref{fig:X} further reveals the ``delay mismatch" in the non-adaptive control law leads to slower convergence after the time point when the exact delay estimate is obtained and the updated input signal reaches the ODE through the transport PDEs. Even though the simulation model, like many practical models that usually include damping, is not an open-loop unstable plant,  the proposed control design still shows improved convergence rates under the proposed delay-adaptive controllers as compared to both the open-loop case and nonadaptive delay-compensated controller. Similarly, it is shown in Figures \ref{fig:z}, \ref{fig:w}, and \ref{fig:v} that the PDE plant states $z(x,t)$, $w(x,t)$ and the actuator state $v(x,t)$ all converge to zero when the system is subject to the proposed delay-adaptive control inputs with $\hat D(0)=0.25$ or $\hat D(0)=1.5$.
\begin{figure}
\centering
\includegraphics[width=9cm]{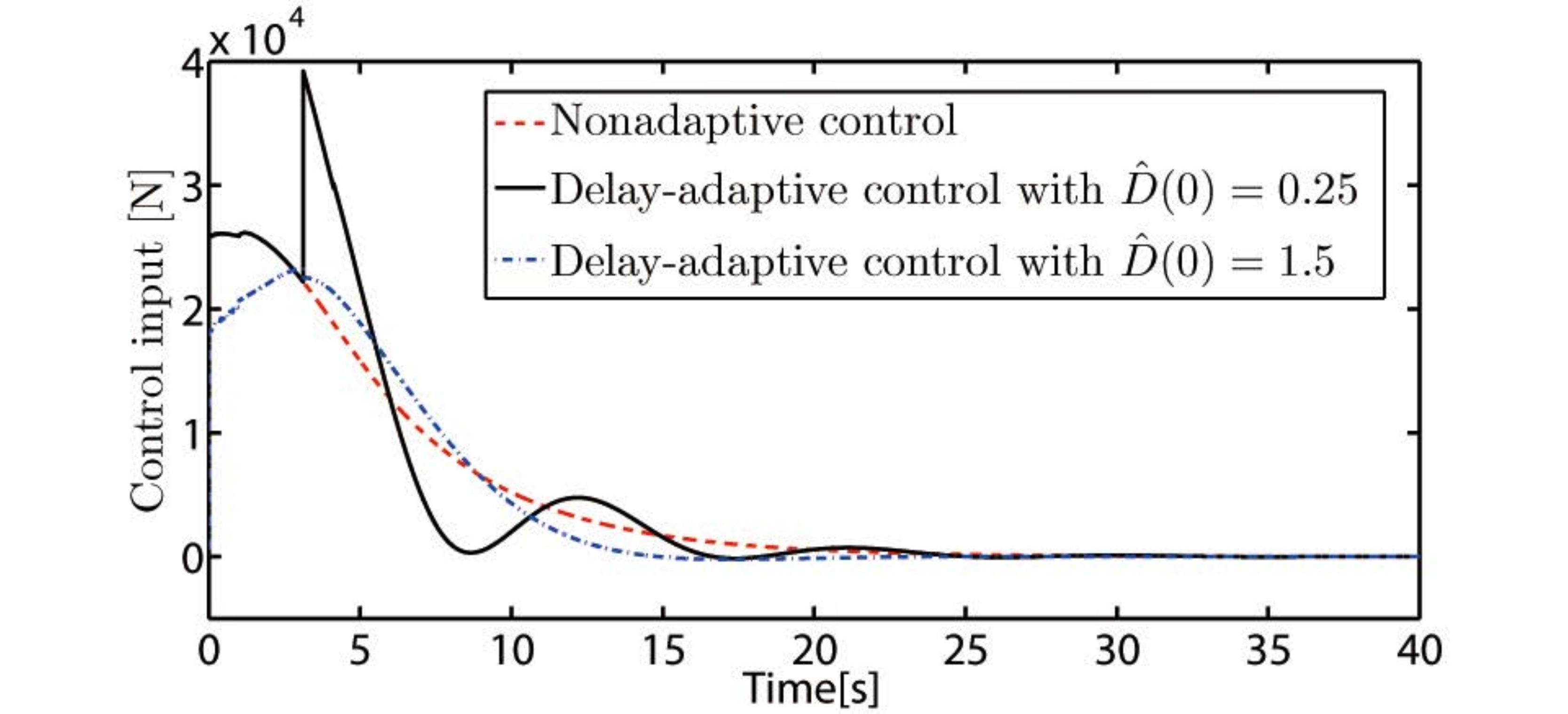}
\caption{The delay-adaptive control input $U_{\rm d}(t)$ with $\hat D(0)=0.25$ or $\hat D(0)=1.5$ and the nonadaptive control input $U_0(t)$.}
\label{fig:control}
\end{figure}
\begin{figure}
\centering
\includegraphics[width=9cm]{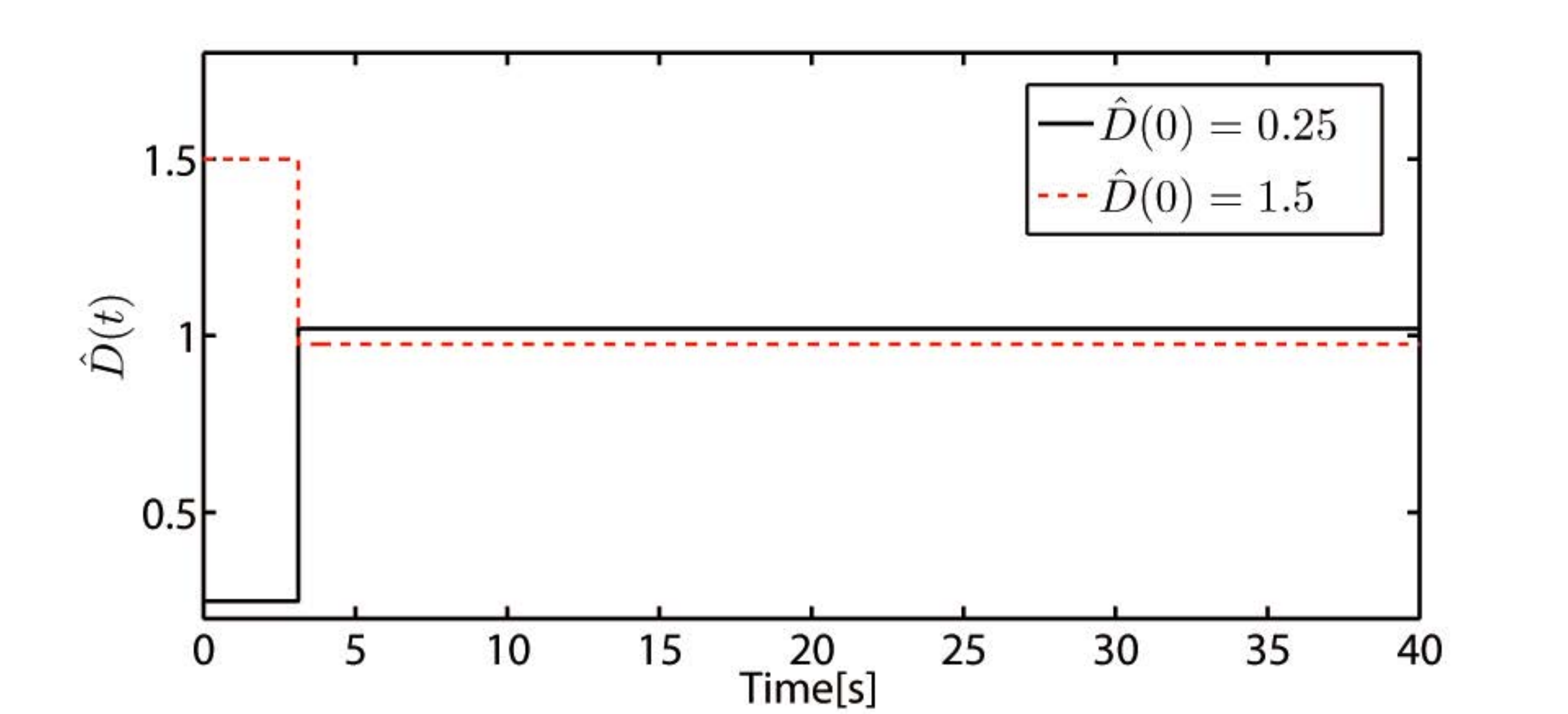}
\caption{Estimate of the unknown delay $D$ under the initial estimate $\hat D(0)=0.25$ or $\hat D(0)=1.5$.}
\label{fig:estimate}
\end{figure}
\begin{figure}
\centering
\includegraphics[width=9cm]{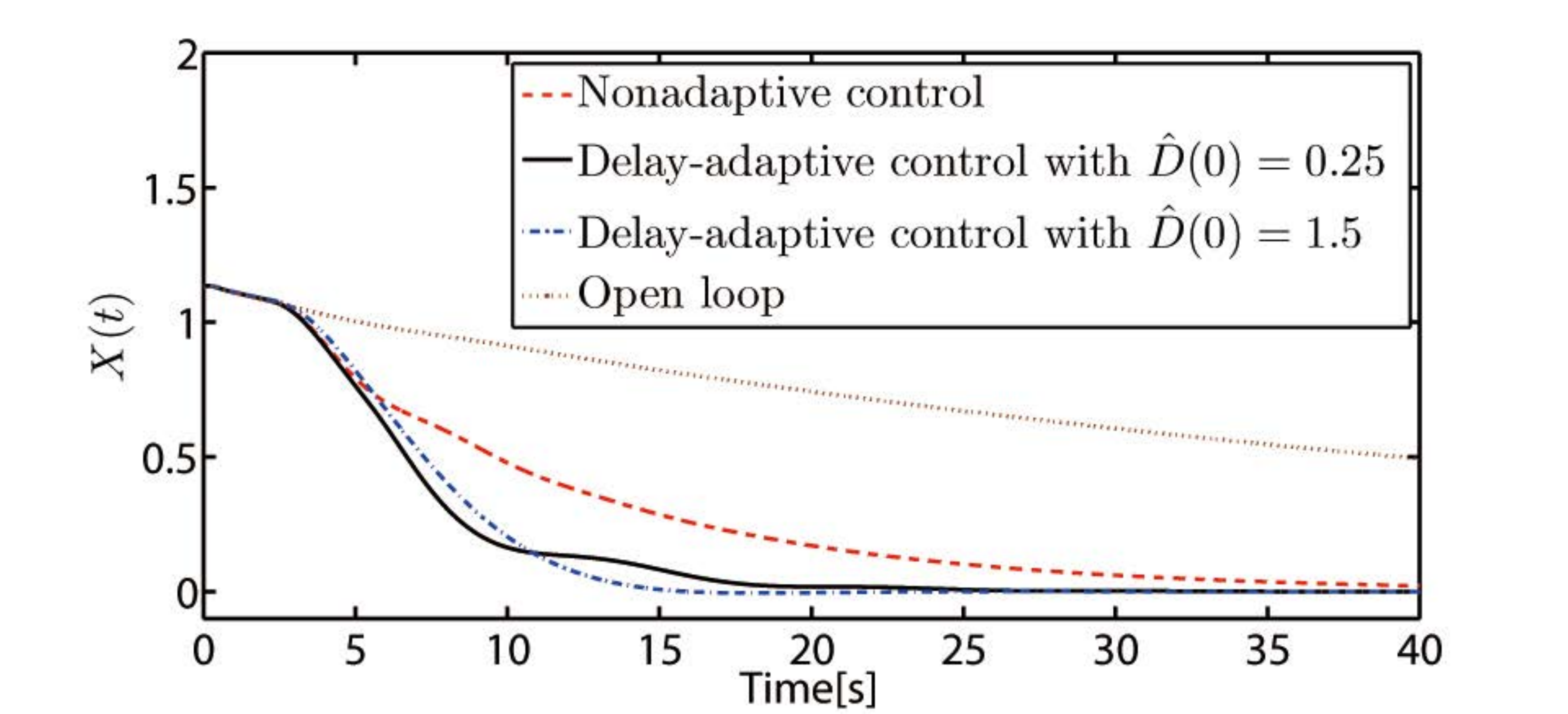}
\caption{The evolution of $X(t)$ under the delay-adaptive control $U_{\rm d}(t)$ with $\hat D(0)=0.25$ or $\hat D(0)=1.5$ and the nonadaptive control $U_0(t)$.}
\label{fig:X}
\end{figure}
\begin{figure}
\begin{minipage}{0.49\linewidth}
  \centerline{\includegraphics[width=4.8cm]{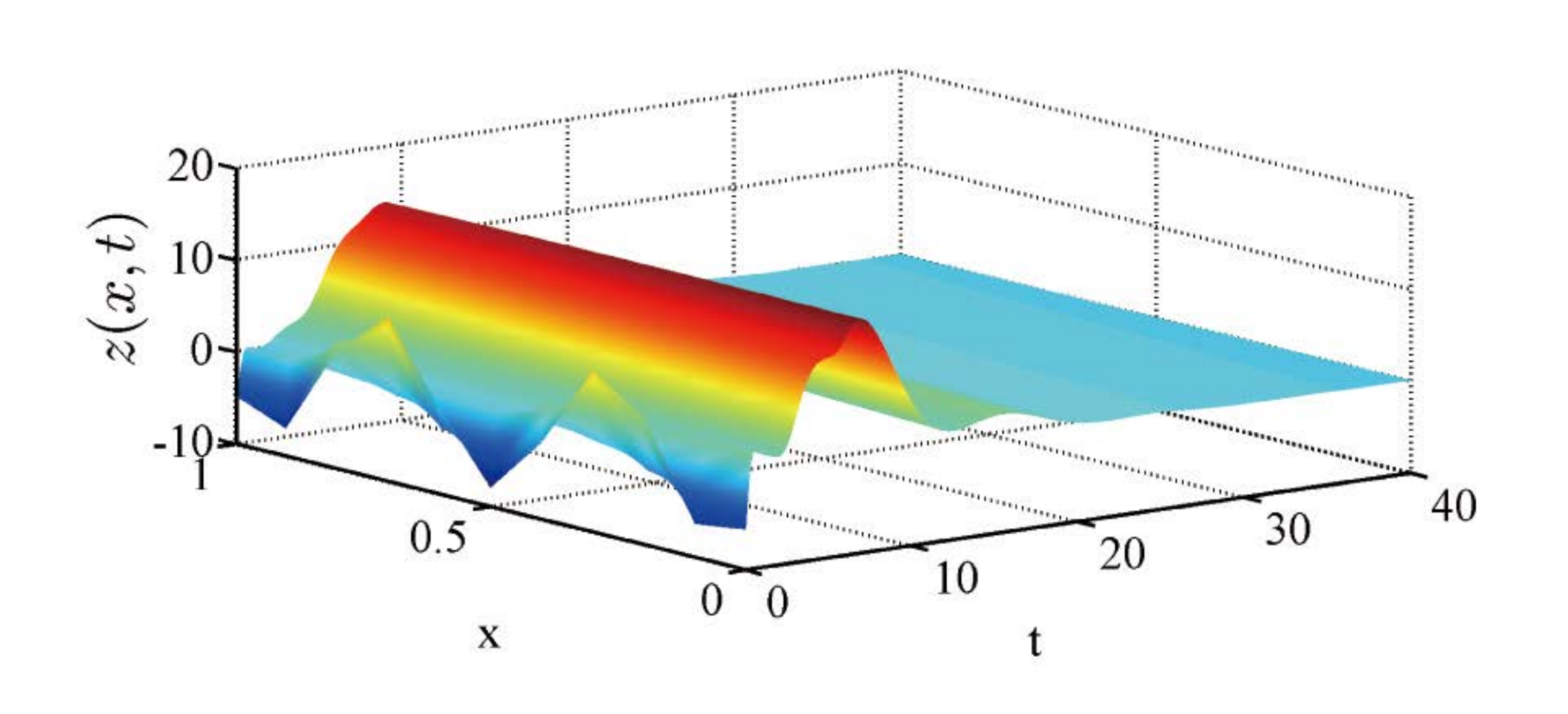}}
  \centerline{(a) $\hat D(0)=0.25$}
\end{minipage}
\hfill
\begin{minipage}{.49\linewidth}
  \centerline{\includegraphics[width=4.8cm]{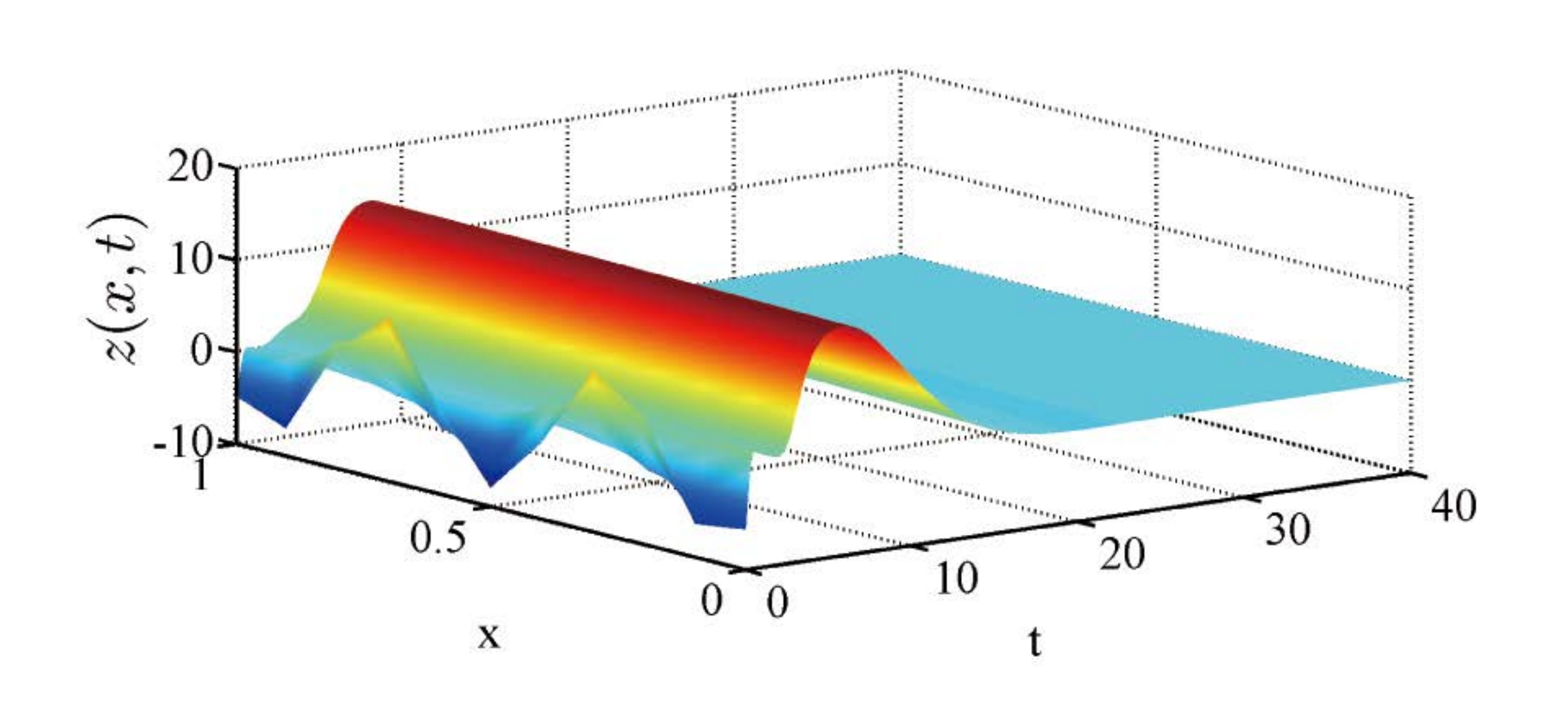}}
  \centerline{(b) $\hat D(0)=1.5$}
\end{minipage}
\caption{The evolution of the plant state $z(x,t)$ under the delay-adaptive control $U_{\rm d}(t)$ with $\hat D(0)=0.25$ or $\hat D(0)=1.5$.}
\label{fig:z}
\end{figure}
\begin{figure}[!h]
\begin{minipage}{0.49\linewidth}
  \centerline{\includegraphics[width=4.8cm]{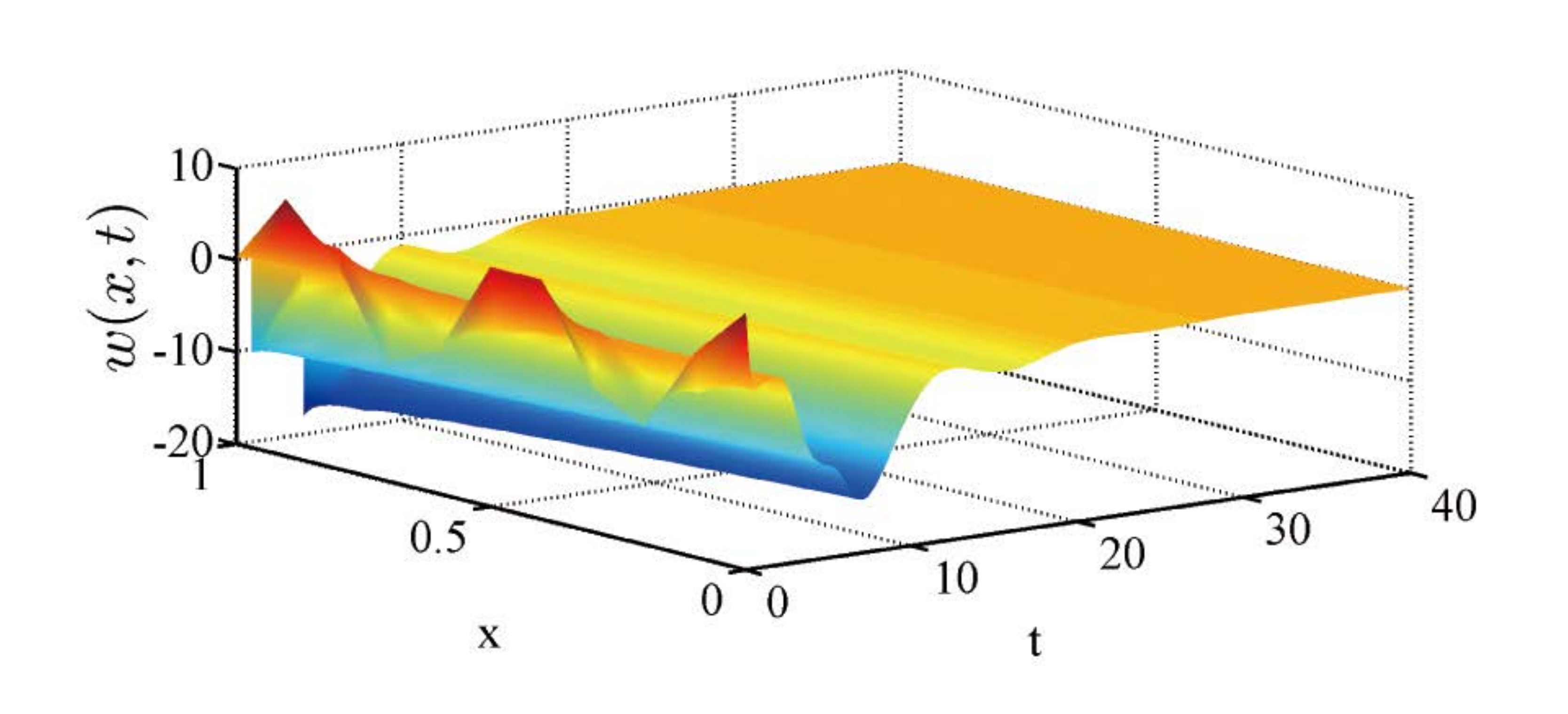}}
  \centerline{(a) $\hat D(0)=0.25$}
\end{minipage}
\hfill
\begin{minipage}{.49\linewidth}
  \centerline{\includegraphics[width=4.8cm]{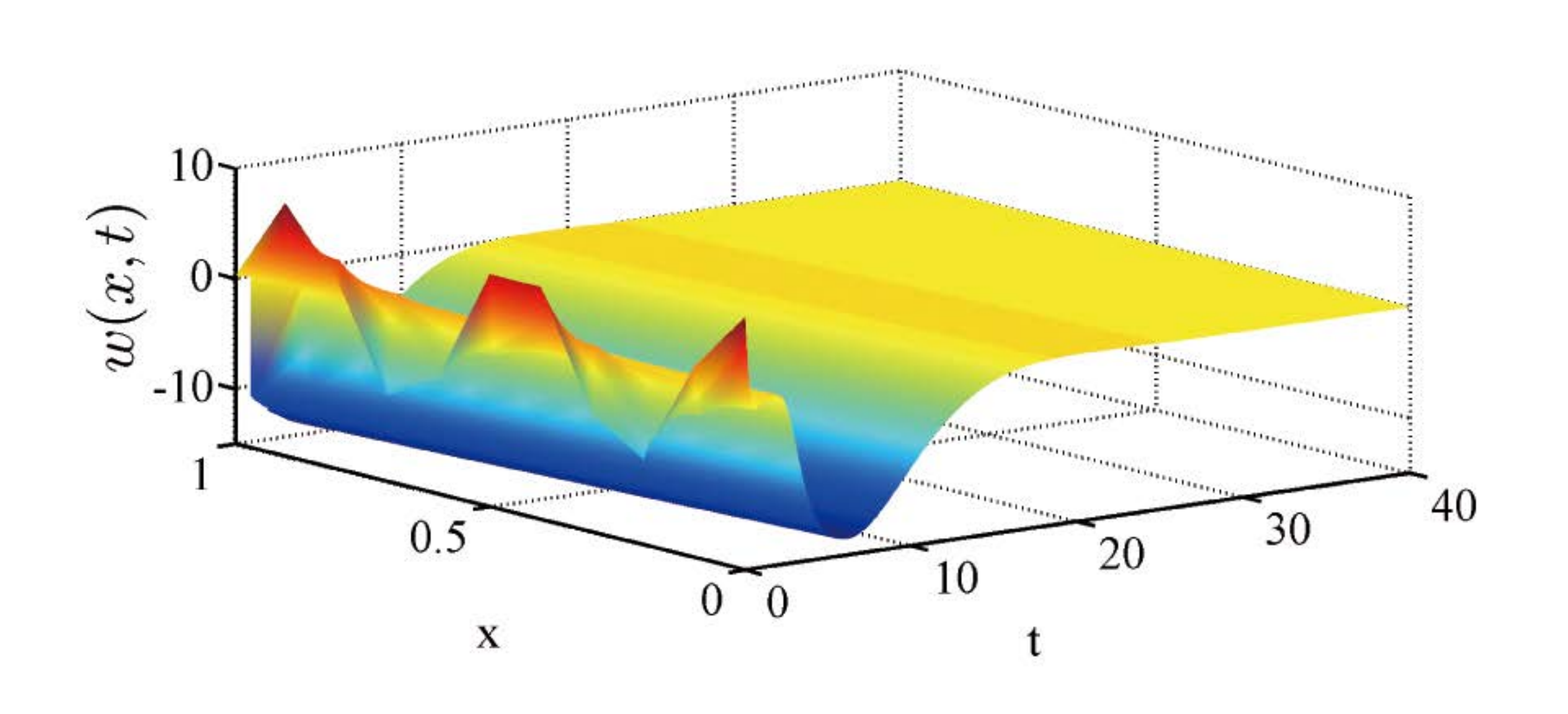}}
  \centerline{(b) $\hat D(0)=1.5$}
\end{minipage}
\caption{The evolution of the plant state $w(x,t)$ under the delay-adaptive control $U_{\rm d}(t)$ with $\hat D(0)=0.25$ or $\hat D(0)=1.5$.}
\label{fig:w}
\end{figure}
\begin{figure}[!h]
\begin{minipage}{0.49\linewidth}
  \centerline{\includegraphics[width=4.8cm]{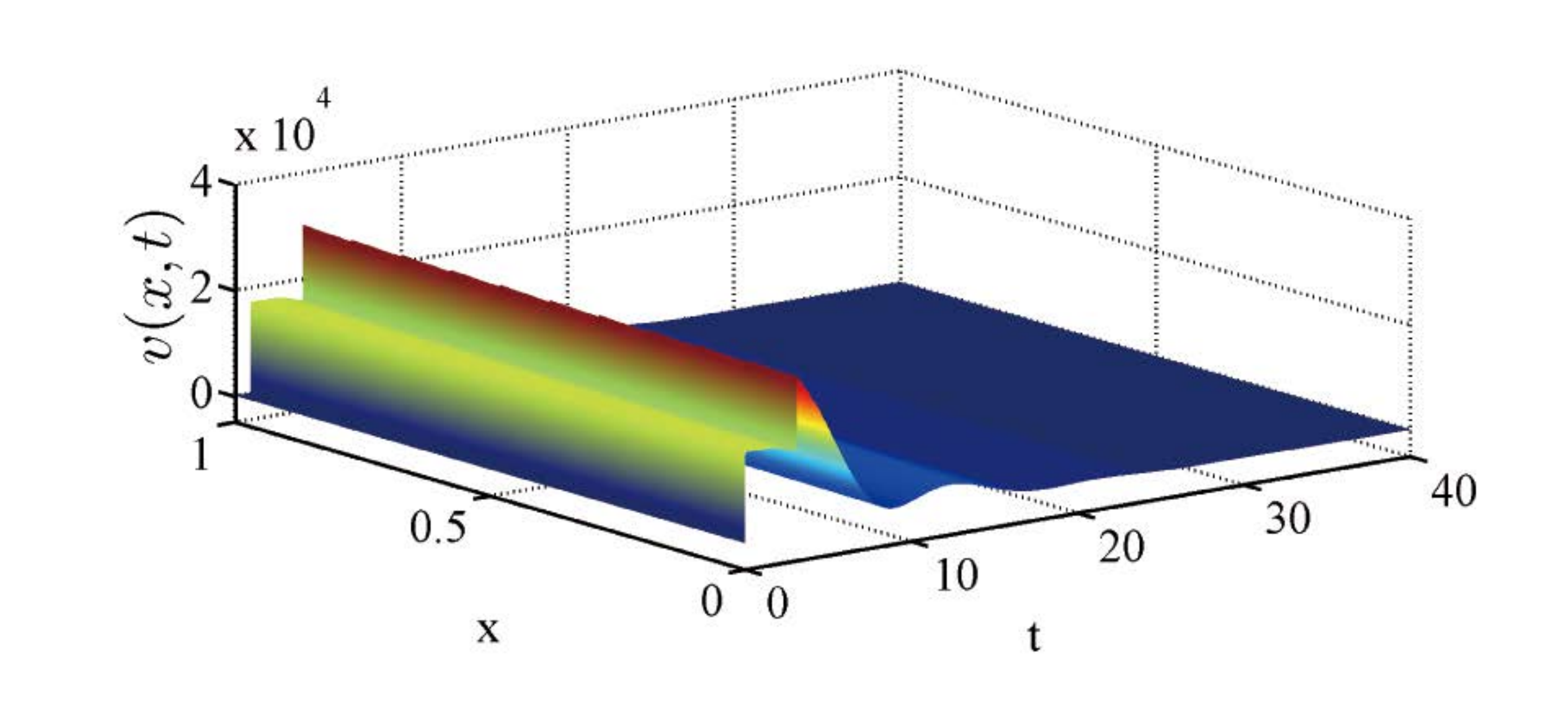}}
  \centerline{(a) $\hat D(0)=0.25$}
\end{minipage}
\hfill
\begin{minipage}{.49\linewidth}
  \centerline{\includegraphics[width=4.8cm]{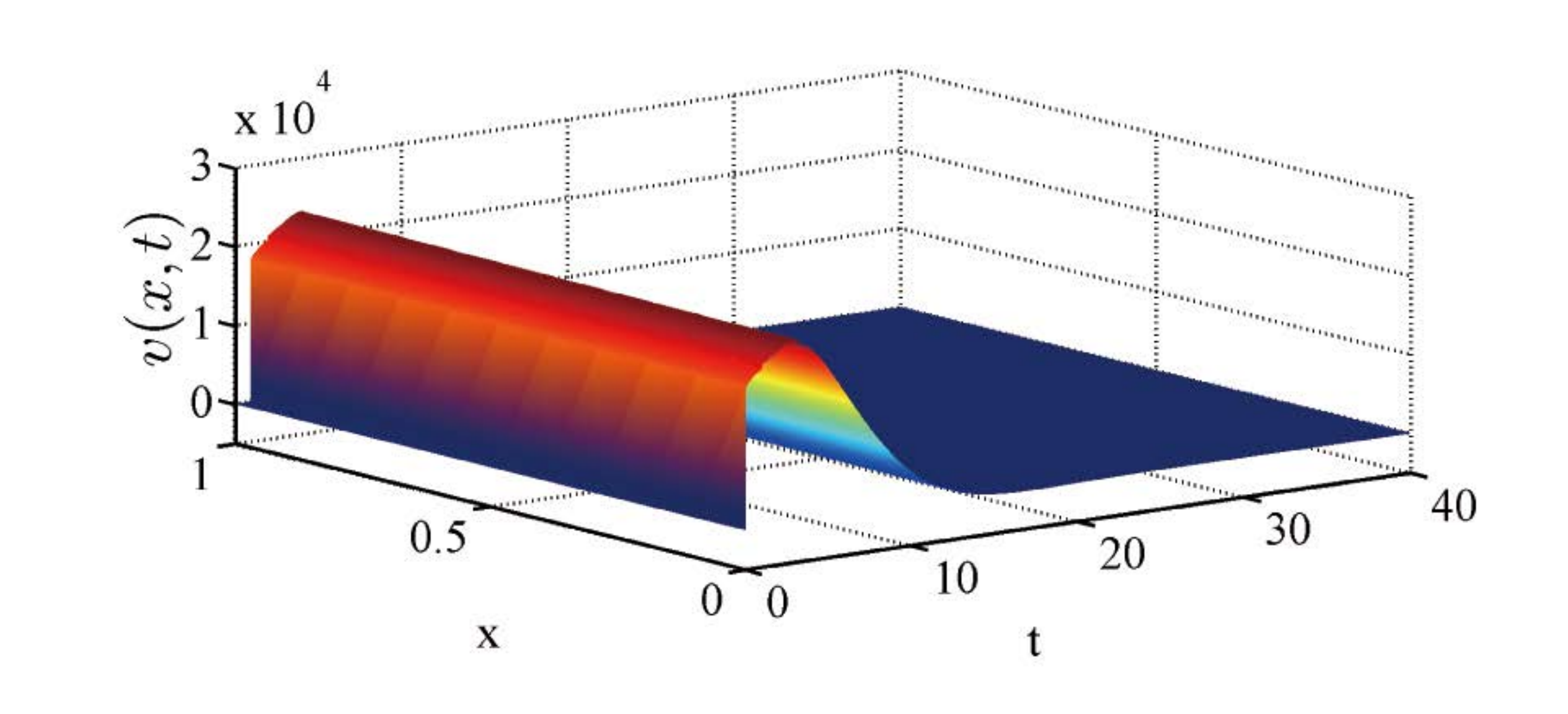}}
  \centerline{(b) $\hat D(0)=1.5$}
\end{minipage}
\caption{The evolution of the actuator state $v(x,t)$ under the delay-adaptive control $U_{\rm d}(t)$ with $\hat D(0)=0.25$ or $\hat D(0)=1.5$.}
\label{fig:v}
\end{figure}

It is easy to obtain the  oscillation energy of the cable in DCV $\frac{\rho}{2}\|u_t(\cdot,t)\|^2+\frac{T_0}{2}\|u_x(\cdot,t)\|^2
=\frac{\rho}{8}\|w(\cdot,t)+z(\cdot,t)\|^2+\frac{\rho}{8}\|w(\cdot,t)-z(\cdot,t)\|^2$ by recalling \eqref{eq:t11}--\eqref{eq:barx}. Therefore, it is known from the results $z(x,t)$ and $w(x,t)$ in Figures \ref{fig:z} and \ref{fig:w} that the oscillation energy of the cable decreases to zero fast under the proposed delay-adaptive controller. One can also observe from Figure \ref{fig:X} that the regulation performance of the ODE, i.e., the payload, is satisfied.
\section{Conclusion and future work}\label{sec:conclusion}
 In this paper, we have proposed a delay-adaptive
control scheme for a $2\times 2$
hyperbolic PDE-ODE system, where the input delay is arbitrarily large and unknown. The controller consists
of a nominal delay-compensated control law, a batch least-squares identifier for the unknown delay, and a triggering mechanism to determine the
update times of the identifier. We have proved that the proposed control
guarantees: {1) the avoidance of Zeno phenomenon; 2) the identification of the unknown boundary input delay in finite time in most situations; 3) the exponential regulation of both the  plant and  the actuator states
to zero. The effectiveness of the proposed
design is verified by numerical simulation in the control application of a deep-sea construction vessel subject to input delay. This paper only deals with the state-feedback adaptive-delay control design for coupled hyperbolic PDEs whose actuator states and plant states are measurable.
In our future work, the output-feedback control design with unmeasurable actuator and plant states will be considered.}
\section{Appendix}
\subsection { Gain kernels PDEs and their associated boundary conditions }  \label{ap:kernel}
\textbf{(a) First-step transformation }

\setcounter{equation}{0}
\renewcommand{\theequation}{A.\arabic{equation}}

The  backstepping transformation \eqref{eq:contran1a} and \eqref{eq:contran1b} lead to the following  PDE-ODE system of kernel conditions for $\varphi,\phi,\Psi,\Phi,\gamma$ and $\lambda$. These conditions are derived by mapping the original plant to the first intermediate system.
\begin{align}
&{q_2}{{\varphi }_y}(x,y) - {q_1}{{\varphi }_x}(x,y)  - ({d_4} - {d_1})\varphi (x,y)\notag\\&- {d_2}\phi (x,y)=0,\label{1}\\
&{  {q_1}{{\phi }_x}(x,y) {+} {q_1}{{\phi }_y}(x,y)} {+} {d_3}\varphi (x,y)=0,\\
& {q_2}{{\Psi }_x}(x,y)  - {q_1}{{\Psi }_y}(x,y) + ({d_4} - {d_1})\Psi (x,y)\notag\\& - {d_3}\Phi (x,y)=0,\\
&{{q_2}{\Phi _x}(x,y)+{q_2}{\Phi _y}(x,y)  - {d_2}\Psi (x,y)}=0,\\
&  {q_1}{\gamma}'(x) +{\gamma}(x)(A- {d_1}I_n) + {q_1}{C}\phi (x,0)=0,\\
& {q_2}\lambda '(x)-{\lambda}(x)(A - {d_4}I_n) - {q_1}{C}\Psi (x,0)=0,\label{eq:kerf1}
\end{align}
with the boundary conditions
\begin{align}
&\varphi (x,x)=\frac{d_2}{{q_1} + {q_2}},\label{eq:ker1}\\
& {q_2}\varphi (x,0) + {q_1}p\phi (x,0)={\gamma}(x)B,\\
&\Psi (x,x)=\frac{-d_3 }{{q_1} + {q_2}} ,\label{eq:Psi}\\
& {q_2}\Phi (x,0) + {q_1}p\Psi (x,0)= \lambda (x)B,\\
&{\lambda (0)=K^T},\label{eq:kerf}\\
& {{\gamma}(0)=C-pK^T},    \label{eq:ker6}
\end{align}
where  $I_n$ is an identity matrix with dimension $n$.

Similarly, the boundary conditions of  the gain kernels associated with the inverse   backstepping transformation \eqref{eq:Icontran1a}, \eqref{eq:Icontran1b}, namely, $\bar\varphi,\bar\phi,\bar\gamma,\bar\Psi,\bar\Phi$ and $\bar\lambda$ are given by
\begin{align}
&{q_2}{{\bar \Psi }_x}(x,y) - {q_1}{{\bar \Psi }_y}(x,y)  + ({d_4} - {d_1})\bar \Psi (x,y)\notag\\& + {d_3}\bar \phi (x,y)=0,\label{eq:Iker1}\\
&  {q_1}{{\bar \phi }_x}(x,y) +{q_1}{{\bar \phi }_y}(x,y)- {d_2}\bar \Psi (x,y) =0,\\
&{q_2}{{\bar \varphi }_y}(x,y) - {q_1}{{\bar \varphi }_x}(x,y)  - ({d_4} - {d_1})\bar \varphi (x,y)\notag\\& + {d_2}\bar \Phi (x,y)= 0,\\
&{q_2}{\bar \Phi _y}(x,y) + {q_2}{\bar \Phi _x}(x,y) + {d_3}\bar \varphi (x,y) = 0,\\
& {q_1}{\bar\gamma}' (x) -{\bar\gamma}(x)(A+BK^T + {d_1}I_n) - {d_2}\bar\lambda(x) =0,\\
& {q_2}\bar\lambda '(x)+\bar\lambda (x)(A+BK^T+d_4I_n) +{d_3}{\bar\gamma}(x) =0,
\end{align}
with the boundary conditions
\begin{align}
 &\bar \Psi (x,x)=-\frac{d_3}{{q_1} + {q_2}},\\
 &   {q_1}p\bar \phi (x,0)+{q_2}\bar \varphi (x,0) ={\bar\gamma}(x)B,\\
 & \bar \varphi (x,x) =\frac{{d_2}}{{q_1} + {q_2}},\\
 & {q_2}\bar \Phi (x,0) + {q_1}p\bar \Psi (x,0)=\bar\lambda (x)B,\\
 &{\bar\lambda (0)=-K^T},\\
&{{\bar\gamma}(0) = pK^T-{C}.}\label{eq:Ikerf}
\end{align}

The set of equations   \eqref{1}--\eqref{eq:ker6} and \eqref{eq:Iker1}--\eqref{eq:Ikerf} are well-known for  coupled linear heterodirectional hyperbolic
PDE-ODE systems, and their well-posedness has been proved in Theorem 4.1 of \cite{Meglio2017Stabilization}.

\textbf{(b) Second-step transformation }

The gain kernels $K_1,K_2$ and $\eta$ are defined below:
\begin{align}
& dK_{1x}(x,y)+q_1K_{1y}(x,y)=-d_1K_1(x,y),\label{eq:K1con1}\\
&dK_{2x}(x,y)-q_2K_{2y}(x,y)=-d_4K_2(x,y),\label{eq:K2con1}\\
&{d}\eta'(x)A_{\rm m}^{-1}+\eta(x)=0,\label{eq:etacon1}
\end{align}
with the boundary conditions
\begin{align}
& K_1(1,y)=\frac{1}{c_0}{\bar\Psi}(1,y)-\frac{1}{c_0}q{\bar\phi}(1,y),\label{eq:K1con2}\\
& K_2(1,y)=\frac{1}{c_0}{\bar\Phi}(1,y)-\frac{1}{c_0}q{\bar\varphi}(1,y),\label{eq:K2con2}\\
 &K_1(x,1)=\frac{qq_2}{q_1}K_2(x,1),\label{eq:K2con3}\\
 &q_1pK_1(x,0)+q_2K_2(x,0)=\eta(x)B,\label{eq:K1con3}\\
&\eta(1)=- \frac{1}{c_0}q\bar\gamma (1) +  \frac{1}{c_0}\bar\lambda (1).\label{eq:etacon2}
\end{align}
The proof of well-posedness of \eqref{eq:K1con1}--\eqref{eq:etacon2} is given in Lemma 2 in \cite{AuriolACC2022}. To derive  conditions \eqref{eq:K1con1}--\eqref{eq:etacon2}, one needs to consider  \eqref{eq:targ8} and \eqref{eq:secondtr}. Hence, \eqref{eq:tar9} holds straightforwardly under the conditions \eqref{eq:K1con2}, \eqref{eq:K2con2}, \eqref{eq:etacon2}. Taking the time and spatial derivatives of \eqref{eq:secondtr}, inserting the results into \eqref{eq:tar10}, recalling \eqref{eq:targ5}--\eqref{eq:targ4}, \eqref{eq:tar9}, one obtains
\begin{align}
&u_t(x,t)+du_x(x,t)-q_2K_2(x,1)c_0u(1,t)\notag\\
=&v_t(x,t)+\int_0^1 K_1(x,y)\alpha_t(y,t)dy+\int_0^1 K_2(x,y)\beta_t(y,t)dy\notag\\&+d v_x(x,t)+d\int_0^1 K_{1x}(x,y)\alpha(y,t)dy\notag\\&+d\int_0^1 K_{2x}(x,y)\beta(y,t)dy\notag\\&+\eta(x)\dot{X}(t)+d\eta'(x){X}(t)-q_2K_2(x,1)c_0u(1,t)\notag\\
=&v_t(x,t)-q_1\int_0^1 K_1(x,y)\alpha_x(y,t)dy\notag\\&+d_1\int_0^1 K_1(x,y)\alpha(y,t)dy\notag\\&+q_2\int_0^1 K_2(x,y)\beta_x(y,t)dy+d_4\int_0^1 K_2(x,y)\beta(y,t)dy\notag\\&+d v_x(x,t)+d\int_0^1 K_{1x}(x,y)\alpha(y,t)dy\notag\\&+d\int_0^1 K_{2x}(x,y)\beta(y,t)dy+\eta(x)(A_{\rm m}{X}(t)+B\beta(0,t))\notag\\&+d\eta'(x){X}(t)-q_2K_2(x,1)c_0u(1,t)\notag\\
%=&-q_1K_1(x,1)\alpha(1,t)+q_1K_1(x,0)\alpha(0,t)\notag\\&+q_1\int_0^1 K_{1y}(x,y)\alpha(y,t)dy+d_1\int_0^1 K_1(x,y) \alpha(y,t)dy\notag\\&+q_2K_2(x,1)\beta(1,t)+(\eta(x)B-q_2K_2(x,0))\beta(0,t)\notag\\&-q_2\int_0^1 K_{2y}(x,y)\beta(y,t)dy+d_4\int_0^1 K_2(x,y)\beta(y,t)dy\notag\\&+d\int_0^1 K_{1x}(x,y)\alpha(y,t)dy+d\int_0^1 K_{2x}(x,y)\beta(y,t)dy\notag\\&+(A_{\rm m}\eta(x)+d\eta'(x)){X}(t)-q_2K_2(x,1)c_0u(1,t)\notag\\
=&\left(q_2K_2(x,1)q-q_1K_1(x,1)\right)\alpha(1,t)\notag\\&+\int_0^x \left(q_1K_{1y}(x,y)+d_1K_1(x,y)+ dK_{1x}(x,y)\right)\alpha(y,t)dy\notag\\&-\left(q_2K_2(x,0)+q_1K_1(x,0)p-\eta(x)B\right)\beta(0,t)\notag\\&+\int_0^x \left(d_4K_2(x,y)-q_2K_{2y}(x,y)+dK_{2x}(x,y)\right)\beta(y,t)dy\notag\\& +\left(\eta(x)A_{\rm m}+d\eta'(x)\right){X}(t)=0.\label{eq:utux}
\end{align}
The necessary and sufficient conditions for  \eqref{eq:utux} to hold are given as  \eqref{eq:K1con1}--\eqref{eq:etacon1}, \eqref{eq:K2con3}, \eqref{eq:K1con3}.

\textbf{(c) Third-step transformation }

The derivation of the gain kernels PDE   $R$ and $R^{I}$ is performed as follows. Substituting the time and spatial derivatives of \eqref{eq:third} into \eqref{eq:tar10} and  recalling \eqref{eq:tar12}--\eqref{eq:tar14}, we have
\begin{align}
&u_t(x,t)+du_x(x,t)-q_2K_2(x,1)c_0u(1,t)\notag\\
=&\hat u_t(x,t)+\int_x^1 R(x,y) \hat u_t(y,t)dy+d\hat u_x(x,t)\notag\\&+d\int_x^1 R_x(x,y) \hat u(y,t)dy-dR(x,x) \hat u(x,t)\notag\\&-q_2K_2(x,1)c_0\hat u(1,t)\notag\\
=&-d\int_x^1 R(x,y) \hat u_x(y,t)dy+d\int_x^1 R_x(x,y) \hat u(y,t)dy\notag\\&-dR(x,x) \hat u(x,t)-q_2K_2(x,1)c_0\hat u(1,t)\notag\\
=&-dR(x,1)\hat u(1,t)+dR(x,x)\hat u(x,t)\notag\\&+d\int_x^1 R_y(x,y) \hat u(y,t)dy+d\int_x^1 R_x(x,y) \hat u(y,t)dy\notag\\&-dR(x,x) \hat u(x,t)-q_2K_2(x,1)c_0\hat u(1,t)\notag\\
=&-(dR(x,1)+q_2K_2(x,1)c_0)\hat u(1,t)\notag\\&+d\int_x^1 (R_x(x,y)+R_y(x,y)) \hat u(y,t)dy\notag\\
=&0.\label{eq:thirdcon}
\end{align}
For \eqref{eq:thirdcon} to hold, the following equality must be satisfied:
\begin{align}
&R_x(x,y)+R_y(x,y)=0,\\
&dR(x,1)=-q_2c_0K_2(x,1),
\end{align}
which obviously admits a unique solution $$R(x,y)=-\frac{q_2c_0}{d}K_2(x-y+1,1).$$

Similarly, substituting the time and spatial derivatives of \eqref{eq:thirdinv} into \eqref{eq:tar13} and  recalling \eqref{eq:tar10}, we have
\begin{align*}
&\hat u_t(x,t)+d\hat u_x(x,t)\\
=&u_t(x,t)+\int_x^1 P(x,y) u_t(y,t)dy+du_x(x,t)\notag\\&+d\int_x^1 P_x(x,y) u(y,t)dy-dP(x,x) u(x,t)\\
=&q_2K_2(x,1)c_0u(1,t)-d\int_x^1 P(x,y) u_x(y,t)dy\notag\\&+\int_x^1 P(x,y)q_2K_2(y,1)c_0dyu(1,t)\notag\\&+d\int_x^1 P_x(x,y) u(y,t)dy-dP(x,x) u(x,t)\\
=&q_2K_2(x,1)c_0u(1,t)-dP(x,1) u(1,t)+dP(x,x) u(x,t)\notag\\&+d\int_x^1 P_y(x,y) u(y,t)dy+\int_x^1 P(x,y)q_2K_2(y,1)c_0dyu(1,t)\\&+d\int_x^1 P_x(x,y) u(y,t)dy-dP(x,x) u(x,t)\\
=&\big(q_2K_2(x,1)c_0-dP(x,1)\notag\\&+\int_x^1 P(x,y)q_2K_2(y,1)c_0dy\big) u(1,t)\\
&+d\int_x^1 (P_y(x,y)+P_x(x,y)) u(y,t)dy=0.
\end{align*}
The equation above suggests that the kernel function  $P$ in the inverse transformation \eqref{eq:thirdinv} satisfies the following PDE with the corresponding boundary value:
\begin{align}
&P_y(x,y)-P_x(x,y)=0,\\
&P(x,1)=\frac{q_2}{d}K_2(x,1)c_0+\frac{1}{d}\int_x^1 P(x,y)q_2K_2(y,1)c_0dy,
\end{align}
whose well-posedness can be obtained by the method of characteristics.
\subsection{Expressions of the controller gain functions $M_1$, $M_2$, $M_3$, $M_4$}
\setcounter{equation}{0}
\renewcommand{\theequation}{B.\arabic{equation}}

The functions $M_1$, $M_2$, $M_3$, and $M_4$ are given as follows,
\begin{align*}
M_1(y)=&\int_0^1 R(0,s)  K_1(s,y)ds-K_1(0,y)\notag\\&+\int_0^1R(0,s_1) \int_{s_1}^1  P(s_1,s) K_1(s,y)dsd{s_1}\\&-\int_y^1 \bigg[\int_0^1 R(0,y)  K_1(y,s)dy-K_1(0,s)\notag\\&+\int_0^1R(0,s_1) \int_{s_1}^1  P(s_1,y) K_1(y,s)dyd{s_1}\bigg] {\phi}(s,y)ds\\&-\int_y^1 \bigg[\int_0^1 R(0,y)   K_2(y,s)dy-K_2(0,s)\notag\\&+\int_0^1 R(0,s_1) \int_{s_1}^1 P(s_1,y) K_2(y,s)dyd{s_1}\bigg] {\Psi}(s,y) ds,\\
M_2(y)=&-\int_y^1 \bigg[\int_0^1 R(0,y)  K_1(y,s)dy-K_1(0,s)\notag\\&+\int_0^1R(0,s_1) \int_{s_1}^1  P(s_1,y) K_1(y,s)dyd{s_1}\bigg]{\varphi}(s,y)ds\\&+\int_0^1 R(0,s)   K_2(s,y)ds-K_2(0,y)\notag\\&+\int_0^1 R(0,s_1) \int_{s_1}^1 P(s_1,s) K_2(s,y)dsd{s_1}\\&-\int_y^1\bigg[\int_0^1 R(0,y)   K_2(y,s)dy-K_2(0,s)\notag\\&+\int_0^1 R(0,s_1) \int_{s_1}^1 P(s_1,y) K_2(y,s)dyd{s_1}\bigg] {\Phi}(s,y)ds,\\
M_3(y)=&R(0,y)+\int_0^yR(0,s)  P(s,y)ds,\\
M_4=&\int_0^1 R(0,y) \eta(y)dy-\eta(0)\\&+\int_0^1 R(0,y) \int_y^1 P(y,s) \eta(s)dsdy\\& - \int_0^1\bigg[\int_0^1 R(0,s)   K_2(s,y)ds-K_2(0,y)\notag\\&+\int_0^1 R(0,s_1) \int_{s_1}^1 P(s_1,s) K_2(s,y)dsd{s_1}\bigg]\lambda (y)dy\\&-\int_0^1 \bigg[\int_0^1 R(0,s)  K_1(s,y)ds-K_1(0,y)\notag\\&+\int_0^1R(0,s_1) \int_{s_1}^1  P(s_1,s) K_1(s,y)dsd{s_1}\bigg]\gamma (y)dy,
\end{align*}
where $K_1,K_2,\eta,R,P$ are parametrized by the unknown delay  $D=\frac{1}{d}$ according to the conditions defined in Appendix A.
\subsection{Norms equivalence between the original and the target systems' states}
\setcounter{equation}{0}
\renewcommand{\theequation}{C.\arabic{equation}}

From \eqref{eq:contran1a}--\eqref{eq:Icontran1b}, \eqref{eq:third}, \eqref{eq:thirdinv}, we get
\begin{align}
\|\alpha (\cdot,t)\|^2\le& \eta_1\bigg(\|z(\cdot,t)\|^2+\|w(\cdot,t)\|^2+|X(t)|^2\bigg),\label{eq:e1}\\
\|\beta (\cdot,t)\|^2\le& \eta_2\bigg(\|z(\cdot,t)\|^2+\|w(\cdot,t)\|^2+|X(t)|^2\bigg),\label{eq:e2}\\
\|z (\cdot,t)\|^2\le& \eta_3\bigg(\|\alpha(\cdot,t)\|^2+\|\beta(\cdot,t)\|^2+|X(t)|^2\bigg),\\
\|w(\cdot,t)\|^2\le& \eta_4\bigg(\|\alpha(\cdot,t)\|^2+\|\beta(\cdot,t)\|^2+|X(t)|^2\bigg),\\
\|u(x,t)\|^2\le& \eta_5\|\hat u(x,t)\|^2,\label{eq:e5}\\
\|\hat u(x,t)\|^2\le& \eta_6\|u(x,t)\|^2,\label{eq:e6}
\end{align}
where
\begin{align*}
\eta_1=&4\bigg(1+\int_0^1\int_0^x {\phi}(x,y)^2dydx+\int_0^1\int_0^x {\varphi}(x,y)^2dydx\notag\\&+\int_0^1\gamma (x)^2dx\bigg),\\
\eta_2=&4\bigg(1+\int_0^1\int_0^x {\Psi}(x,y)^2dydx+\int_0^1\int_0^x {\Phi}(x,y)^2dydx\notag\\&+\int_0^1\lambda (x)^2dx\bigg),\\
\eta_3=&4\bigg(1+\int_0^1\int_0^x {\bar\phi}(x,y)^2dydx+\int_0^1\int_0^x {\bar\varphi}(x,y)^2dydx\notag\\&+\int_0^1\bar\gamma (x)^2dx\bigg),\\
\eta_4=&4\bigg(1+\int_0^1\int_0^x {\bar\Psi}(x,y)^2dydx+\int_0^1\int_0^x {\bar\Phi}(x,y)^2dydx\notag\\&+\int_0^1\bar\lambda (x)^2dx\bigg),\\
\eta_5=&2\bigg(1+\int_0^1\int_x^1 R(x,y)^2dydx\bigg),\\
\eta_6=&2\bigg(1+\int_0^1\int_x^1 P(x,y)^2dydx\bigg).
\end{align*}
Recalling \eqref{eq:secondtr}, together with \eqref{eq:e1}, \eqref{eq:e2}, \eqref{eq:e5}, \eqref{eq:e6}, yields
\begin{align}
||v(\cdot,t)||^2\le&4\bigg(\eta_5+\int_0^1\int_0^1 K_1(x,y)^2dydx\notag\\&+\int_0^1\int_0^1 K_2(x,y)^2dydx+\int_0^1\eta(x)^2dx\bigg)\big(\|\hat u(\cdot,t)\|^2\notag\\&+||\alpha(\cdot,t)||^2+||\beta(\cdot,t)||^2+|{X}(t)|^2\big),\\
||\hat u(\cdot,t)||^2\le&4\eta_6(\eta_1+\eta_2+1)\bigg(1+\int_0^1\int_0^1 K_1(x,y)^2dydx\notag\\&+\int_0^1\int_0^1 K_2(x,y)^2dydx+\int_0^1\eta(x)^2dx\bigg)\big(||v(\cdot,t)||^2\notag\\
&+\|z(\cdot,t)\|^2+\|w(\cdot,t)\|^2+|X(t)|^2\big).
\end{align}
Defining
\begin{align}
\bar\Omega(t)=\|\alpha[t]\|^2+\|\beta[t]\|^2+\|\hat u[t]\|^2+ {X}(t)^2\label{eq:bom}
\end{align}
one obtains
\begin{align}\label{equi}
\xi_1 \Omega(t)\le \bar\Omega(t)\le \xi_2 \Omega(t)
\end{align}
where
\begin{align}
\xi_1=&1/\bigg(1+\eta_3+\eta_4+4\eta_5+4\int_0^1\int_0^1 K_1(x,y)^2dydx\notag\\&+4\int_0^1\int_0^1 K_2(x,y)^2dydx+4\int_0^1\eta(x)^2dx\bigg),\label{eq:xi1}\\
\xi_2=&1+\eta_1+\eta_2\notag\\&+4\eta_6(\eta_1+\eta_2+1)\bigg(1+\int_0^1\int_0^1 K_1(x,y)^2dydx\notag\\&+\int_0^1\int_0^1 K_2(x,y)^2dydx+\int_0^1\eta(x)^2dx\bigg).\label{eq:xi2}
\end{align}

\end{document}